\documentclass[10pt,letterpaper]{amsart}
\usepackage[utf8]{inputenc}
\usepackage{amsmath,amssymb,amsthm,mathtools}
\usepackage[margin=1in]{geometry}
\usepackage{enumitem}
\usepackage{shuffle}
\usepackage[hidelinks]{hyperref}

\newcommand{\Law}{\mathrm{Law}}
\newcommand{\Prob}{\mathbb{P}}
\DeclareMathOperator{\Unif}{Unif}
\newcommand{\E}{\mathbb{E}}
\newcommand{\R}{\mathbb{R}}
\newcommand{\supp}{\operatorname{supp}}
\DeclareMathOperator{\TV}{TV}
\DeclareMathOperator{\sep}{sep}
\DeclareMathOperator{\im}{im}

\theoremstyle{plain}
\newtheorem{theorem}{Theorem}[section]
\newtheorem{lemma}[theorem]{Lemma}
\newtheorem{proposition}[theorem]{Proposition}
\newtheorem{corollary}[theorem]{Corollary}
\theoremstyle{definition}
\newtheorem{definition}[theorem]{Definition}
\newtheorem{example}[theorem]{Example}
\newtheorem{remark}[theorem]{Remark}

\title{Explicit cutoff profiles for colored top-\texorpdfstring{$m$}{m}-to-random shuffles}
\author{Ivan Z. Feng}
\address{Department of Mathematics, University of Southern California, Los Angeles, CA 90089-2532, USA}
\email{ifeng@usc.edu}
\urladdr{https://dornsife.usc.edu/ivan/}
\date{May 27, 2026}

\begin{document}

\begin{abstract}
We study \(p\)-colored top-\(m\)-to-random on the wreath product \(G_{n,p}=C_p\wr S_n\), with \(m\) fixed. Using the Nakano--Sadahiro--Sakurai basis elements \(B_m\), we obtain exact nested-set occupancy mixtures and reduce the likelihood ratio to the single statistic \(L_p\). This yields exact formulas for separation and \(L^\infty(U)\), and exact one-dimensional formulas for total variation, \(L^q(U)\) (\(1\le q<\infty\)), \(\chi^2\), and relative entropy. At
\[
k=\Bigl\lfloor \frac{n}{m}(\log n+c)\Bigr\rfloor,
\]
the number of never-chosen labels in the associated \(m\)-subset occupancy model converges in law to \(\mathrm{Poisson}(e^{-c})\), giving the total-variation profile \(f_p(c)\), the separation profile, and the corresponding \(L^q(U)\), \(L^\infty(U)\), \(\chi^2\), and relative-entropy profiles. For \(m=1\) we recover colored top-to-random; for \(p=1\), the total-variation profile reduces to the Diaconis--Fill--Pitman profile. For the reversed chain, we also identify optimal strong stationary times whose tail probabilities are exactly the separation distances.
\end{abstract}

\maketitle
\tableofcontents
\section{Introduction}

Diaconis--Fill--Pitman \cite{DFP} determined the total-variation cutoff profile for classical top-to-random on \(S_n\) at times
\[
k=\lfloor n\log n+cn\rfloor.
\]
For integers \(n,p\ge1\), let
\[
G_{n,p}=C_p\wr S_n\cong C_p^{\,n}\rtimes S_n,
\]
the wreath product of the cyclic group \(C_p\) with the symmetric group \(S_n\); equivalently, \(G_{n,p}\) is the group of \(p\)-colored permutations.
Nakano--Sadahiro--Sakurai \cite{NSS} proved cutoff on the \(n\log n\) scale for colored top-to-random, but did not give a closed profile.
We obtain explicit cutoff profiles for colored top-\(m\)-to-random with fixed $m$, together with exact and asymptotic formulas for separation, $L^\infty(U)$, $L^q(U)$, $\chi^2$, and relative entropy. For \(m=1\) we recover colored top-to-random; for \(m=p=1\), we recover the Diaconis--Fill--Pitman total-variation profile \cite{DFP}. 

A one-step \(p\)-colored top-\(m\)-to-random move removes the top $m$ cards, recolors them by independent uniform elements of $C_p$, applies a uniform random permutation to these $m$ cards, chooses an $m$-element set of final positions uniformly among the \(\binom{n}{m}\) possibilities, and inserts the $m$ cards in those positions. The remaining $n-m$ cards preserve their relative order and color. Section~\ref{sec:prel} gives the group-algebra realization via the basis elements \(B_m\) introduced by Nakano--Sadahiro--Sakurai \cite{NSS}.

By the time-reversal discussion in Section~\ref{sec:prel}, passing to the inverse law does not change total variation or separation. For \(m=1\), this connects colored top-to-random with colored move-to-front. Weighted colored move-to-front chains also appear in the \(p\)-colored Tsetlin-library literature; see \cite{NakagawaNakano}.
For \(p=1,2\) there is a real hyperplane-arrangement realization in which faces act on chambers by the Tits projection \cite{AthanasiadisDiaconis_Hyperplane, BrownDiaconisHyperplane}: \(p=1\) is the braid arrangement of type \(A_{n-1}\), and \(p=2\) the signed arrangement of type \(B_n\); see also \cite{Bjorner2008Libraries}. In these cases, chamber hitting gives an optimal strong stationary time for separation \cite{Nestoridi_HyperplaneSST}.
For \(p\ge3\), the natural reflection arrangement of \(G_{n,p}\) is complex rather than real, so we work algebraically through the basis and nested-set mixtures. Related algebraic frameworks include Brown's semigroup methods \cite{Brown2000Semigroups} and Pang's Hopf-algebraic descent-operator approach \cite{Pang2016Hopf}.

Our starting point is the product expansion for $B_m$ in \cite{NSS}. It yields exact nested-set occupancy mixtures $\widetilde Q_{\mu,p}$ on the decreasing family
\[
A_0\supseteq A_1\supseteq \cdots \supseteq A_n.
\]
The nesting compresses the likelihood ratio to one statistic:
\[
\frac{\widetilde Q_{\mu,p}(x)}{U(x)}=S_\mu(L_p(x)).
\]
At
\[
k=\left\lfloor \frac{n}{m}(\log n+c)\right\rfloor,
\]
the number of never-chosen labels converges in law to \(\operatorname{Poisson}(e^{-c})\). This coupon-collector variable is the source of all cutoff-window profiles in the paper. Once the likelihood ratio has been compressed to \(S_\mu(L_p(x))\), each metric question becomes a one-dimensional question about the distribution of \(L_p\). Thus the same mechanism gives the total-variation profile \(f_p(c)\), the exact separation and \(L^\infty(U)\) formulas, and the \(L^q(U)\), \(\chi^2\), and relative-entropy profiles. In particular, separation is determined by $\min L_p$, with a sharp $p=1$ versus $p\ge2$ dichotomy. Here $U$ denotes the uniform law on $G_{n,p}$. Section~\ref{sec:prel} fixes the Markov-chain and convolution conventions. For a broad treatment of card shuffling, including move-to-front-type models and geometric viewpoints, see Diaconis and Fulman \cite{DiaconisFulmanShuffling}. For background on mixing and cutoff, see \cite{Diaconis96Cutoff, LPW}. For general comparisons among probability metrics, see Gibbs and Su \cite{GibbsSu}. Unlike total variation, separation is not symmetric and hence is not a distance between probability measures; see also \cite{DiaconisSaloffCosteSep}.

For \(0\le u\le n\), let \(A_u\) be the set of colored permutations in which the largest \(u\) labels appear in increasing order and have color \(0\), and let \(\widetilde Q_{n-u,p}\) be the uniform law on \(A_u\). Set
\[
L_p(x):=\max\{u:x\in A_u\},\qquad
Q^{(m)}:=\widetilde Q_{m,p},
\]
and, for a probability mass function \(\mu\) on \(\{0,\dots,n\}\),
\[
\widetilde Q_{\mu,p}:=\sum_{u=0}^{n}\mu(n-u)\,\widetilde Q_{n-u,p}.
\]
In asymptotic statements involving a sequence \((\mu_n)\) of probability mass functions on \(\{0,\dots,n\}\), write
\[
\lambda:=e^{-c},\qquad
r:=p\lambda,\qquad
M_n:=\widetilde Q_{\mu_n,p},
\]
extend \(\mu_n(a)=0\) for \(a\notin\{0,\dots,n\}\), and assume
\begin{equation}\label{eq:intro-Poisson-assump}
\mu_n(n-u)\longrightarrow e^{-\lambda}\frac{\lambda^u}{u!}
\qquad(u\ge0).
\end{equation}
We also write \(\chi^2(M,U):=\left\|\frac{M(\cdot)}{U(\cdot)}-1\right\|_{L^2(U)}^2\) and \(D(M\|U):=\sum_x M(x)\log\frac{M(x)}{U(x)}\) with the convention \(0\log0:=0\).
For asymptotic top-\(m\)-to-random statements with fixed \(m\ge1\) and \(c\in\R\), write
\begin{equation}\label{eq:window-time-top-m}
k_n^{(m)}(c):=\left\lfloor \frac{n}{m}(\log n+c)\right\rfloor.
\end{equation}
For fixed \(m\) and \(c\), this integer is nonnegative for all sufficiently large \(n\), which is the regime used in the asymptotic statements.
When \(m\) is fixed throughout a statement, we also write \(k_n(c)\) for \(k_n^{(m)}(c)\).

For the reversed chain, set
\[
\check Q^{(m)}(g):=Q^{(m)}(g^{-1}).
\]
Let \((Y_j)_{j\ge0}\) be the \(\check Q^{(m)}\)-walk started at \(e\). In its inverse-move realization, each step chooses a uniform \(m\)-subset of labels, moves those labels to the first \(m\) positions in uniform random order, recolors them independently and uniformly in \(C_p\), and leaves all other labels and colors in relative order. Let \(T_j\subseteq[n]\) be the set of labels touched by time \(j\), with \(T_0=\varnothing\), and set
\[
P_j^{(m)}(u):=\Prob(n-|T_j|=u),
\]
and define
\[
\tau_{\mathrm{cov}}:=\min\{j:T_j=[n]\},
\qquad
\tau_{n-1}:=\min\{j:|T_j|\ge n-1\}.
\]

We now state three main theorems.
For the Poisson-regime statements below, write
\[
w^*:=w^*(c,p):=\min\left\{\ell\ge0:\ e^{-\lambda}\sum_{u=0}^{\ell}r^u>1\right\},
\qquad
a_\ell:=\frac{1}{\ell!\,p^\ell}-\frac{1}{(\ell+1)!\,p^{\ell+1}},
\qquad
s(\ell):=e^{-\lambda}\sum_{u=0}^{\ell}r^u.
\]
Set
\[
f_p(c):=
1-e^{-\lambda}\sum_{u=0}^{w^*-1}\frac{\lambda^u}{u!}
+\frac{e^{-\lambda}\sum_{u=0}^{w^*-1}r^u-1}{w^*!\,p^{w^*}},
\qquad
s_p(c):=
\begin{cases}
1-e^{-\lambda}(1+\lambda), & p=1,\\[4pt]
1-e^{-\lambda}, & p\ge2,
\end{cases}
\]
\[
H_{q,p}(c):=\sum_{\ell=0}^{\infty}a_\ell\,|s(\ell)-1|^q
\quad(1\le q<\infty),
\qquad
h_p(c):=\sum_{\ell=0}^{\infty}a_\ell\,s(\ell)\log s(\ell),
\]
\[
g_p(c):=H_{2,p}(c)=
\begin{cases}
\displaystyle \frac{r+1}{r-1}e^{-\lambda(2-r)}-\frac{2}{r-1}e^{-\lambda}-1,
& r\neq1,\\[12pt]
\displaystyle (2\lambda+1)e^{-\lambda}-1,
& r=1,
\end{cases}
\qquad
H_{\infty,p}(c):=
\begin{cases}
\displaystyle \frac{e^{-\lambda}}{1-r}-1, & r<1,\\[10pt]
+\infty, & r\ge1.
\end{cases}
\]
Theorem~\ref{thm:intro-LR} gives the exact one-statistic reduction and the resulting exact formulas for separation, \(L^q(U)\), \(L^\infty(U)\), \(\chi^2\), and relative entropy. Theorem~\ref{thm:intro-engine} gives the general Poisson-regime profiles, with extra domination hypotheses for \(L^q(U)\), relative entropy, and \(L^\infty(U)\). Theorem~\ref{thm:intro-topm} specializes these results to colored top-\(m\)-to-random and identifies optimal strong stationary times for the reversed chain.

\begin{theorem}[Exact reduction to \texorpdfstring{$L_p$}{Lp}]\label{thm:intro-LR}
Let \(\mu\) be a probability mass function on \(\{0,1,\dots,n\}\), and define
\[
\widetilde Q_{\mu,p}:=\sum_{u=0}^{n}\mu(n-u)\,\widetilde Q_{n-u,p},
\qquad
S_\mu(\ell):=\sum_{u=0}^{\ell}\mu(n-u)\,u!\,p^u.
\]
Then
\[
\frac{\widetilde Q_{\mu,p}(x)}{U(x)}=S_\mu(L_p(x))
\qquad(x\in G_{n,p}),
\]
and, for every \(\Phi:[0,\infty)\to\mathbb R\) for which the sums are finite,
\[
\sum_{x\in G_{n,p}}U(x)\,
\Phi\!\left(\frac{\widetilde Q_{\mu,p}(x)}{U(x)}\right)
=
\sum_{\ell=0}^{n}U\{L_p=\ell\}\,\Phi\!\bigl(S_\mu(\ell)\bigr).
\]
Moreover,
\[
\sep(\widetilde Q_{\mu,p},U)=
\begin{cases}
1-\mu(n)-\mu(n-1), & p=1,\\[4pt]
1-\mu(n), & p\ge2,
\end{cases}
\]
and
\[
\left\|\frac{\widetilde Q_{\mu,p}(\cdot)}{U(\cdot)}-1\right\|_{L^\infty(U)}
=S_\mu(n)-1
=\sum_{u=0}^{n}\mu(n-u)\,u!\,p^u-1.
\]
\end{theorem}

\begin{theorem}[Poisson-regime profiles]\label{thm:intro-engine}
Fix \(p\ge1\) and \(c\in\R\). Assume \eqref{eq:intro-Poisson-assump}. Then
\[
\|M_n-U\|_{\TV}\longrightarrow f_p(c),
\qquad
\sep(M_n,U)\longrightarrow s_p(c).
\]
If in addition there exists \(B\ge1\) such that, for all sufficiently large \(n\) and all \(u\ge0\),
\[
0\le \mu_n(n-u)\,u!\,p^u\le B^u,
\]
then, for every \(1\le q<\infty\),
\[
\left\|\frac{M_n(\cdot)}{U(\cdot)}-1\right\|_{L^q(U)}^q
\longrightarrow H_{q,p}(c),
\qquad
D(M_n\|U)\longrightarrow h_p(c).
\]
In particular,
\[
\chi^2(M_n,U)\longrightarrow H_{2,p}(c)=g_p(c).
\]
Assume further, in the case \(r<1\), that there exists a summable \((b_u)_{u\ge0}\) such that, for all sufficiently large \(n\) and all \(u\ge0\),
\[
0\le \mu_n(n-u)\,u!\,p^u\le b_u.
\]
Then, with convergence understood in the extended interval \([0,\infty]\),
\[
\left\|\frac{M_n(\cdot)}{U(\cdot)}-1\right\|_{L^\infty(U)}
\longrightarrow H_{\infty,p}(c).
\]
\end{theorem}

\begin{theorem}[Cutoff-window profiles for colored top-\texorpdfstring{$m$}{m}-to-random and reversed-chain optimal strong stationary times]\label{thm:intro-topm}

Fix \(c\in\R\) and integers \(p\ge1\) and \(m\ge1\), with \(m\) independent of \(n\). Let \(k=k_n(c)\). Then
\[
\sep\bigl((Q^{(m)})^{*k},U\bigr)=s_p(c)+o(1),
\]
\[
\left\|\frac{(Q^{(m)})^{*k}(\cdot)}{U(\cdot)}-1\right\|_{L^q(U)}^q
=
H_{q,p}(c)+o(1)
\qquad(1\le q<\infty),
\]
\[
D\bigl((Q^{(m)})^{*k}\|U\bigr)=h_p(c)+o(1),
\]
and, with convergence in the extended interval \([0,\infty]\),
\[
\left\|\frac{(Q^{(m)})^{*k}(\cdot)}{U(\cdot)}-1\right\|_{L^\infty(U)}
\longrightarrow H_{\infty,p}(c).
\]
In particular,
\[
\bigl\|(Q^{(m)})^{*k}-U\bigr\|_{\TV}
=
f_p(c)+o(1),
\qquad
\chi^2\bigl((Q^{(m)})^{*k},U\bigr)
=
g_p(c)+o(1).
\]
Hence the total-variation cutoff is at \(\frac{n}{m}\log n\) with window \(\frac{n}{m}\).

Furthermore, with the reversed-chain notation above, if \(p\ge2\), then \(\tau_{\mathrm{cov}}\) is an optimal strong stationary time and
\[
\Prob(\tau_{\mathrm{cov}}>j)
=
1-P_j^{(m)}(0)
=
\sep\bigl((Q^{(m)})^{*j},U\bigr)
\qquad(j\ge0).
\]
If \(p=1\), then \(\tau_{n-1}\) is an optimal strong stationary time and
\[
\Prob(\tau_{n-1}>j)
=
1-P_j^{(m)}(0)-P_j^{(m)}(1)
=
\sep\bigl((Q^{(m)})^{*j},U\bigr)
\qquad(j\ge0).
\]
\end{theorem}

\section*{Acknowledgments}

The author thanks Jason Fulman for suggesting this problem and for providing references and valuable feedback.

\section{Preliminaries}\label{sec:prel}

This section fixes the convolution, metric, and group-action conventions used throughout the paper. It also recalls the Nakano--Sadahiro--Sakurai basis elements \(B_m\), their supports, and the exact \(m=1\) expansion.

\paragraph{\textbf{Markov chains and convolution.}}
Let \(X\) be a finite group and let \(Q\) be a probability measure on \(X\). We let successive steps act by right multiplication: if \((G_t)_{t\ge1}\) are i.i.d.\ with law \(Q\), then
\[
X_t=X_{t-1}G_t\qquad(t\ge1),
\qquad\text{so}\qquad
X_k=X_0G_1\cdots G_k.
\]
Hence the transition kernel is
\begin{equation}\label{eq:intro-P-from-Q-DFP}
P(x,y)=\Prob(X_{t+1}=y\mid X_t=x)
=\Prob(xG_{t+1}=y)
=Q(x^{-1}y).
\end{equation}

For probability measures \(M,N\) on \(X\), define
\begin{equation}\label{eq:intro-conv-DFP}
(M*N)(z):=\sum_{u\in X}M(zu^{-1})\,N(u).
\end{equation}
This is exactly the law of a product: if $G\sim M$ and $H\sim N$ are independent, then for each $z\in X$,
\[
\Prob(GH=z)=\sum_{u\in X}\Prob(H=u)\Prob(G=zu^{-1})
=\sum_{u\in X}N(u)M(zu^{-1})
=(M*N)(z);
\]
i.e., \(\Law(GH)=M*N\). 

This convolution is the evolution rule for distributions under the kernel $P$. Indeed, for any probability measure $\mu$ on $X$,
\[
(\mu P)(y):=\sum_{x\in X}\mu(x)P(x,y)
=\sum_{x\in X}\mu(x)Q(x^{-1}y)
=\sum_{u\in X}\mu(yu^{-1})Q(u)
=(\mu*Q)(y),
\]
so $\mu_{t+1}=\mu_tP=\mu_t*Q$.
Starting from \(X_0=e\), we get \(\mu_0=\delta_e\) and hence
\[
\mu_k=Q^{*k},
\qquad
Q^{*k}:=\underbrace{Q*\cdots*Q}_{k\text{ times}},
\qquad
Q^{*k}(z)=P^k(e,z).
\]
More generally
\begin{equation}\label{eq:intro-Pk-Qk-DFP}
P^k(x,y)=\Prob(X_k=y\mid X_0=x)=\Prob(G_1\cdots G_k=x^{-1}y)=Q^{*k}(x^{-1}y)\qquad(k\ge0).
\end{equation}

If we identify a measure \(M\) with the group-algebra element \(\sum_{x\in X}M(x)\,x\in\mathbb RX\), convolution becomes multiplication in \(\mathbb RX\):
\[
MN=\sum_{z\in X}(M*N)(z)\,z.
\]
In particular, \(Q^{*k}=Q^k\) in \(\mathbb RX\).

\medskip

\paragraph{\textbf{Distances and cutoff.}}
Let \(U\) be the uniform measure on \(X\). For probability measures \(M,N\) on \(X\), set
\[
\|M-N\|_{\TV}:=\frac12\sum_{x\in X}|M(x)-N(x)|,
\qquad
\sep(M,U):=1-\min_{x\in X}\frac{M(x)}{U(x)}.
\]
For \(1\le q\le\infty\), define
\[
\|f\|_{L^q(U)}:=
\begin{cases}
\left(\sum_{y\in X}|f(y)|^qU(y)\right)^{1/q}, & 1\le q<\infty,\\[4pt]
\max_{y\in X}|f(y)|, & q=\infty.
\end{cases}
\]
Write
\[
q_k(x,y):=\frac{P^k(x,y)}{U(y)},
\qquad
d^{(q)}(k):=\max_{x\in X}\|q_k(x,\cdot)-1\|_{L^q(U)},
\qquad
d(k):=\max_{x\in X}\|P^k(x,\cdot)-U\|_{\TV},
\]
and
\[
s_x(k):=\sep(P^k(x,\cdot),U)
=1-\min_{y\in X}q_k(x,y),
\qquad
s(k):=\max_{x\in X}s_x(k).
\]

\begin{proposition}[Comparison of \(L^q\), total variation, and separation]\label{prop:distance-comparison}
For every \(k\ge0\) and \(x\in X\),
\[
\|P^k(x,\cdot)-U\|_{\TV}\le s_x(k)\le \|q_k(x,\cdot)-1\|_{L^\infty(U)}.
\]
Consequently,
\[
d(k)\le s(k)\le d^{(\infty)}(k).
\]
Also,
\[
d^{(1)}(k)=2d(k)\le d^{(2)}(k)\le d^{(\infty)}(k),
\]
and hence
\[
\frac12\,d^{(1)}(k)\le s(k)\le d^{(\infty)}(k).
\]
\end{proposition}

\begin{proof}
The bound \(\|P^k(x,\cdot)-U\|_{\TV}\le s_x(k)\) is \cite[Lemma~6.16]{LPW}. Also,
\[
s_x(k)=\max_{y\in X}\bigl(1-q_k(x,y)\bigr)
\le \max_{y\in X}|q_k(x,y)-1|
=\|q_k(x,\cdot)-1\|_{L^\infty(U)}.
\]
Taking maxima over \(x\) gives \(d(k)\le s(k)\le d^{(\infty)}(k)\). Finally,
\[
d^{(1)}(k)=2d(k)\le d^{(2)}(k)\le d^{(\infty)}(k)
\]
is from \cite[Section~4.7]{LPW}.
\end{proof}

\begin{proposition}[Translation invariance on a group]\label{prop:distance-translation}
For every \(k\ge0\), \(x\in X\), and \(1\le q\le\infty\),
\[
\|P^k(x,\cdot)-U\|_{\TV}=\|Q^{*k}-U\|_{\TV},
\]
\[
\|q_k(x,\cdot)-1\|_{L^q(U)}
=
\left\|\frac{Q^{*k}(\cdot)}{U(\cdot)}-1\right\|_{L^q(U)},
\]
and
\[
s_x(k)=\sep(Q^{*k},U).
\]
Therefore
\[
d(k)=\|Q^{*k}-U\|_{\TV},
\qquad
d^{(q)}(k)=\left\|\frac{Q^{*k}(\cdot)}{U(\cdot)}-1\right\|_{L^q(U)},
\qquad
s(k)=\sep(Q^{*k},U).
\]
\end{proposition}

\begin{proof}
By \eqref{eq:intro-Pk-Qk-DFP},
\[
P^k(x,y)=Q^{*k}(x^{-1}y),
\qquad
q_k(x,y)=\frac{Q^{*k}(x^{-1}y)}{U(y)}.
\]
Since \(y\mapsto x^{-1}y\) is a bijection and \(U\) is uniform, the three identities follow immediately after the change of variables \(z=x^{-1}y\).
\end{proof}

For a random walk on a finite group with increment law \(Q\), the reversed law
\[
\check Q(g):=Q(g^{-1})
\]
generates the time-reversed walk. Moreover, the total-variation distance to stationarity is unchanged under time reversal for group walks; see  \cite[Section~4.6]{LPW} and \cite[Lemma~3.16]{Aldous1983}. In the present setting we also have
\[
\check Q^{*k}(x^{-1})=Q^{*k}(x),
\]
so the same invariance holds for separation and all \(L^q(U)\) distances:
\[
\|Q^{*k}-U\|_{\TV}=\|\check Q^{*k}-U\|_{\TV},
\qquad
\sep(Q^{*k},U)=\sep(\check Q^{*k},U),
\]
and
\[
\left\|\frac{Q^{*k}(\cdot)}{U(\cdot)}-1\right\|_{L^q(U)}
=
\left\|\frac{\check Q^{*k}(\cdot)}{U(\cdot)}-1\right\|_{L^q(U)}.
\]

A (total-variation) cutoff at times \(t_n\) with window \(b_n\) means that there exists a nonincreasing function
\(f:\R\to[0,1]\) with \(f(c)\to1\) as \(c\to-\infty\) and \(f(c)\to0\) as \(c\to+\infty\) such that
\[
\|Q^{*\lfloor t_n+cb_n\rfloor}-U\|_{\TV}\longrightarrow f(c)
\qquad(n\to\infty)
\]
for each fixed \(c\in\R\). The limit \(f\) is the \emph{cutoff profile}.
\medskip

\paragraph{\textbf{The colored permutation group \(G_{n,p}\).}}
Fix \(n,p\ge1\), let \(C_p=\mathbb Z/p\mathbb Z\) written additively, and set
\[
G_{n,p}:=C_p\wr S_n\cong C_p^{\,n}\rtimes S_n.
\]
Following \cite{NSS}, we realize \(G_{n,p}\) as \(C_p^{\,n}\times S_n\) with multiplication
\[
(t,\tau)(s,\sigma)=(\sigma t+s,\tau\sigma),
\qquad
\sigma t:=(t_{\sigma(1)},\dots,t_{\sigma(n)}).
\]

Let \(\mathbb RG_{n,p}\) be the real group algebra and \(\mathbb QG_{n,p}\subseteq \mathbb RG_{n,p}\) its \(\mathbb Q\)-subalgebra. We identify any real-valued function \(M:G_{n,p}\to\mathbb R\) with
\[
M=\sum_{x\in G_{n,p}}M(x)\,x\in\mathbb RG_{n,p}.
\]
All conventions above now apply with \(X=G_{n,p}\).

Let \(U=U_{n,p}\) denote the uniform law on \(G_{n,p}\):
\[
U(x)=\frac{1}{|G_{n,p}|}=\frac{1}{p^n n!}.
\]
This is stationary for every right random walk on \(G_{n,p}\).

\medskip

\paragraph{\textbf{Words and shuffle product.}}
We identify \((s,\sigma)\in G_{n,p}\) with the word
\[
(s_1,\sigma(1))(s_2,\sigma(2))\cdots(s_n,\sigma(n))
\]
over the alphabet \(C_p\times[n]\). Conversely, any word
\[
(c_1,i_1)(c_2,i_2)\cdots(c_n,i_n),
\qquad (i_1,\dots,i_n)\in S_n,
\]
corresponds to the unique \((s,\sigma)\in G_{n,p}\) with
\[
\sigma(t)=i_t,
\qquad
s_t=c_t
\qquad(1\le t\le n).
\]
Thus \(\sigma(t)\) is the label in position \(t\), and \(\sigma^{-1}(i)\) is the position of label \(i\).

Let \(\varepsilon\) be the empty word. The shuffle product \(\shuffle\) is defined by
\[
u\shuffle\varepsilon=\varepsilon\shuffle u:=u,
\qquad
(ua)\shuffle(vb):=(u\shuffle vb)\,a+(ua\shuffle v)\,b,
\]
for words \(u,v\) and letters \(a,b\). Intuitively, \(u\shuffle v\) is the formal sum of all interleavings of \(u\) and \(v\) preserving the relative order within each word.

For \(0\le m\le n\), set
\[
W_{m,n}:=(0,m+1)(0,m+2)\cdots(0,n).
\]
Intuitively, $W_{m,n}$ is the ``frozen tail'': the labels $m+1,\dots,n$ appear in increasing order and all have color $0$.

\medskip

\paragraph{\textbf{The Nakano--Sadahiro--Sakurai basis elements \(B_m\).}}

For \(0\le m\le n\), define \(B_m\in\mathbb QG_{n,p}\) by
\[
B_m:=
\begin{cases}
\mathrm{id}, & m=0,\\[4pt]
\displaystyle \sum_{\alpha\in G_{m,p}}\alpha\shuffle W_{m,n}, & 1\le m\le n-1,\\[10pt]
\displaystyle \sum_{\alpha\in G_{n,p}}\alpha, & m=n.
\end{cases}
\]
Here \(G_{m,p}\) is identified with colored words on labels \(1,\dots,m\). In particular,
\[
B_1=\sum_{c\in C_p}(c,1)\shuffle W_{1,n}.
\]

Intuitively, \(B_m\) is the sum of all colored words in which the labels \(m+1,\dots,n\) appear in increasing order and all have color \(0\); the labels \(1,\dots,m\) may appear anywhere, in any order, with arbitrary colors. Equivalently, for \(1\le m\le n-1\),
\[
\supp(B_m)=\Bigl\{x=(s,\sigma)\in G_{n,p}:\ 
\sigma^{-1}(m+1)<\cdots<\sigma^{-1}(n),\
s_{\sigma^{-1}(m+1)}=\cdots=s_{\sigma^{-1}(n)}=0
\Bigr\}.
\]
Also,
\[
\supp(B_0)=\{\mathrm{id}\},
\qquad
\supp(B_n)=G_{n,p}.
\]
\begin{lemma}\label{lem:Ba-indicator}
For each \(0\le a\le n\) and \(x\in G_{n,p}\), we have the coefficient of $x$:
\[
B_a(x)\in\{0,1\}.
\]
Equivalently,
\[
B_a(x)=\mathbf 1_{\{x\in\supp(B_a)\}}
\qquad(x\in G_{n,p}).
\]
\end{lemma}

\begin{proof}
The letters appearing in any word of \(\alpha\shuffle W_{a,n}\) are all distinct, since their labels are \(1,\dots,n\). Hence a final word uniquely determines the positions occupied by the letters of \(\alpha\), their order, and their colors, so it determines \(\alpha\) uniquely. Therefore no group element appears with multiplicity \(>1\) in the sum defining \(B_a\), and the coefficient of any \(x\in G_{n,p}\) is either \(0\) or \(1\).
\end{proof}

For \(0\le a\le n\), let
\[
|B_a|:=|\supp(B_a)|,
\qquad
\widetilde Q_{a,p}:=\frac{1}{|B_a|}B_a.
\]
Thus \(\widetilde Q_{a,p}\) is the uniform probability measure on \(\supp(B_a)\). In particular, for \(1\le m\le n\),
\[
Q^{(m)}:=\widetilde Q_{m,p}=\frac{1}{|B_m|}B_m.
\]
For \(m=1\), write
\[
Q:=Q_{n,p}:=Q^{(1)}=\widetilde Q_{1,p}.
\]

\begin{example}[\(n=3,\ p=2\)]
Here \(C_2=\{0,1\}\). We have
\[
W_{0,3}=(0,1)(0,2)(0,3),\quad
W_{1,3}=(0,2)(0,3),\quad
W_{2,3}=(0,3),\quad
W_{3,3}=\varepsilon.
\]
For example,
\[
(1,1)\shuffle(0,2)(0,3)
=(1,1)(0,2)(0,3)+(0,2)(1,1)(0,3)+(0,2)(0,3)(1,1).
\]
Hence
\[
B_1=\sum_{c\in\{0,1\}}
\Bigl((c,1)(0,2)(0,3)+(0,2)(c,1)(0,3)+(0,2)(0,3)(c,1)\Bigr),
\]
so \(|B_1|=6\) and \(Q=\frac16 B_1\). For \(x\in G_{3,2}\),
\[
B_1(x)=
\begin{cases}
1, & x\in\{
(0,1)(0,2)(0,3),\,
(0,2)(0,1)(0,3),\,
(0,2)(0,3)(0,1),\\
&\hphantom{x\in\{}
(1,1)(0,2)(0,3),\,
(0,2)(1,1)(0,3),\,
(0,2)(0,3)(1,1)
\},\\[4pt]
0, & \text{otherwise}.
\end{cases}
\]

Likewise, if \(\alpha=(1,2)(0,1)\in G_{2,2}\), then
\[
\alpha\shuffle(0,3)
=(0,3)(1,2)(0,1)+(1,2)(0,3)(0,1)+(1,2)(0,1)(0,3),
\]
and \(|B_2|=\binom{3}{2}2^2\cdot2!=24\).
\end{example}

\medskip

\paragraph{\textbf{One-step colored top-to-random.}}
Since \(B_1\) has \(n\) insertion positions and \(p\) color choices, it contains exactly \(np\) distinct group elements. Therefore
\[
Q=Q_{n,p}=\frac{1}{np}B_1
\]
is the uniform law on \(\supp(B_1)\), and the \(k\)-step law is \(Q^{*k}=Q^k\).

\begin{proposition}\label{thm:NSS-stirling} (\cite[Theorem~2.2]{NSS})
For every integer \(k\ge0\),
\[
B_1^k=\sum_{a=0}^{\min(n,k)}p^{\,k-a}
\left\{\!\!\begin{matrix}k\\ a\end{matrix}\!\!\right\}B_a,
\]
where \(\left\{\!\!\begin{smallmatrix}k\\ a\end{smallmatrix}\!\!\right\}\) is the Stirling number of the second kind.
\end{proposition}

\begin{corollary}\label{cor:Qk-Ba-exact}
For every integer \(k\ge0\),
\begin{equation}\label{eq:Qk-Ba-exact}
Q^{*k}
=\sum_{a=0}^{n}\frac{1}{n^k p^a}
\left\{\!\!\begin{matrix}k\\ a\end{matrix}\!\!\right\}B_a.
\end{equation}
\end{corollary}

\begin{proof}
Since \(Q=(np)^{-1}B_1\), we have \(Q^{*k}=\frac{1}{(np)^k}B_1^k\).
Now apply Proposition~\ref{thm:NSS-stirling}:
\begin{align*}
Q^{*k}
&=\frac{1}{(np)^k}B_1^k\\
&=\sum_{a=0}^{\min(n,k)} \frac{p^{k-a}}{(np)^k}
\left\{\!\!\begin{matrix}k\\ a\end{matrix}\!\!\right\}B_a\\
&=\sum_{a=0}^{n}\frac{1}{n^k p^a}
\left\{\!\!\begin{matrix}k\\ a\end{matrix}\!\!\right\}B_a,
\end{align*}
because $\left\{\!\!\begin{smallmatrix}k\\ a\end{smallmatrix}\!\!\right\}=0$ for $a>k$.
\end{proof}

\begin{proposition}\label{lem:Ba-size-exact}(\cite[Section~3]{NSS})
For \(0\le a\le n\),
\[
|B_a|=\binom{n}{a}p^a a!=(n)_a\,p^a,
\qquad
(n)_a:=n(n-1)\cdots(n-a+1)\ \ (a\ge1), \qquad (n)_0:=1.
\]
\end{proposition}

\begin{proof}
Choose the $a$ positions of labels $\{1,\dots,a\}$ ($\binom{n}{a}$ ways), permute them ($a!$ ways), and choose their colors ($p^a$ ways); the remaining labels $a+1,\dots,n$ are forced to appear in increasing order with color $0$.
\end{proof}

\section{Nested-set mixtures}\label{sec:mixture-engine}

This section rewrites the shuffle laws as mixtures over the nested sets \(A_u\), introduces the statistic \(L_p\), and reduces likelihood ratios to the one-dimensional function \(S_\mu(L_p)\). The same reduction gives the exact total-variation formula and the Poisson-regime profile.

\subsection{Exact mixture and \texorpdfstring{$A_u$}{Au}}

Recall from Section~\ref{sec:prel} that
\[
\widetilde Q_{a,p}:=\frac{1}{|B_a|}B_a
\qquad(0\le a\le n),
\]
so \(\widetilde Q_{a,p}\) is the uniform probability measure on \(\supp(B_a)\). By Lemma~\ref{lem:Ba-indicator},
\begin{equation}\label{eq:Qbarsupp}
\widetilde Q_{a,p}(x)=\frac{1}{|B_a|}B_a(x)
=\frac{1}{|B_a|}\mathbf 1_{\{x\in\supp(B_a)\}}
\qquad(x\in G_{n,p}).
\end{equation}

By Corollary~\ref{cor:Qk-Ba-exact},
\[
Q^{*k}=\sum_{a=0}^{n}\frac{1}{n^k p^a}
\left\{\!\!\begin{matrix}k\\ a\end{matrix}\!\!\right\}B_a.
\]

\begin{lemma}[Exact occupancy mixture]\label{lem:Qk-occupancy-mixture}
For every \(k\ge0\),
\begin{equation}\label{eq:mixture-u-exact}
Q^{*k}
=\sum_{a=0}^{n}\mu_k(a)\,\widetilde Q_{a,p}
=\sum_{u=0}^{n}P_k(u)\,\widetilde Q_{n-u,p},
\end{equation}
where
\begin{equation}\label{eq:mu-k-a}
\mu_k(a):=\frac{1}{n^k}\binom{n}{a}\left\{\!\!\begin{matrix}k\\ a\end{matrix}\!\!\right\}a!,
\end{equation}
and \(P_k(u)=\Prob(E=u)\) for the usual balls-into-boxes model with \(k\) i.i.d.\ uniform balls in \(n\) boxes, where \(E\) is the number of empty boxes. Equivalently,
\[
P_k(u)=\mu_k(n-u).
\]
\end{lemma}

\begin{proof}
Using \(B_a=|B_a|\widetilde Q_{a,p}\) and Proposition~\ref{lem:Ba-size-exact},
\begin{align}
Q^{*k}
&=\sum_{a=0}^{n}\frac{1}{n^k p^a}
\left\{\!\!\begin{matrix}k\\ a\end{matrix}\!\!\right\}|B_a|\,\widetilde Q_{a,p}\nonumber\\
&=\sum_{a=0}^{n}\frac{1}{n^k p^a}
\left\{\!\!\begin{matrix}k\\ a\end{matrix}\!\!\right\}\binom{n}{a}p^a a!\,\widetilde Q_{a,p}\nonumber\\
&=\sum_{a=0}^{n}\mu_k(a)\,\widetilde Q_{a,p}.
\label{eq:mixture-a-exact}
\end{align}

To interpret \(\mu_k(a)\), let \(A\) be the number of occupied boxes after throwing \(k\) i.i.d.\ uniform balls into \(n\) boxes. Then
\[
\Prob(A=a)=\frac{\#\{f:[k]\to[n]:|\im(f)|=a\}}{n^k}.
\]
Choose which $a$ boxes are occupied ($\binom{n}{a}$ choices),
partition $[k]$ into $a$ nonempty fibers ($\left\{\!\!\begin{smallmatrix}k\\ a\end{smallmatrix}\!\!\right\}$ choices), and label these fibers by the chosen boxes ($a!$ choices).
This gives
\[
\Prob(A=a)=\frac{1}{n^k}\binom{n}{a}\left\{\!\!\begin{matrix}k\\ a\end{matrix}\!\!\right\}a!=\mu_k(a).
\]
If \(E=n-A\), then \(P_k(u)=\Prob(E=u)=\mu_k(n-u)\), and \eqref{eq:mixture-a-exact} becomes \eqref{eq:mixture-u-exact}.
\end{proof}

\begin{example}[Balls into boxes]
If \(n=5\), \(k=6\), and the ball locations are
\[
(2,2,5,1,2,5),
\]
then the occupied boxes are \(\{1,2,5\}\), so \(A=3\) and \(E=2\).
\end{example}

\begin{remark}
There is a bijection
\[
\{f:[k]\to[n]\}\ \longleftrightarrow\ [n]^k,\qquad
f\ \longmapsto\ (f(1),\dots,f(k)),\ \ \ (x_1,\dots,x_k)\ \longmapsto\bigl(t\mapsto x_t\bigr).
\]
Under this identification,
\[
|\im(f)|=\bigl|\{f(1),\dots,f(k)\}\bigr|=\bigl|\{x_1,\dots,x_k\}\bigr|.
\]
Hence, for each $a$,
\[
\#\{f:[k]\to[n]:|\im(f)|=a\}
=
\#\{(x_1,\dots,x_k)\in[n]^k:\ |\{x_1,\dots,x_k\}|=a\}.
\]
That is, counting functions $f:[k]\to[n]$ with $|\im(f)|=a$ is the same as counting length-$k$ words in $[n]^k$ that use exactly $a$ distinct symbols.
\end{remark}

Note that for \(p=1\), \eqref{eq:mixture-u-exact} is exactly the mixture formula \cite[(2.5)]{DFP}. The formulation in \cite{NSS} provides its colored analog. To see that, we first introduce the following.

\begin{definition}\label{def:Au}
For \(0\le u\le n\), define
\[
A_u:=\supp(\widetilde Q_{n-u,p})=\supp(B_{n-u})\subseteq G_{n,p}.
\]
Equivalently, for \(1\le u\le n\) and \(x=(s,\sigma)\in G_{n,p}\),
\begin{equation}\label{eq:Au-condition-exact}
x\in A_u
\iff
\sigma^{-1}(n-u+1)<\cdots<\sigma^{-1}(n)
\ \text{and}\
s_{\sigma^{-1}(n-u+1)}=\cdots=s_{\sigma^{-1}(n)}=0.
\end{equation}
Also,
\[
A_0=\supp(B_n)=G_{n,p},
\qquad
A_{u+1}\subseteq A_u,
\qquad
|A_u|=|B_{n-u}|=\binom{n}{u}p^{\,n-u}(n-u)!.
\]
\end{definition}

\begin{lemma}\label{lem:nested-intersection}
For \(0\le u,\ell\le n\),
\[
A_\ell\cap A_u=A_{\max(\ell,u)},
\qquad
|A_\ell\cap A_u|=|A_{\max(\ell,u)}|.
\]
\end{lemma}

\begin{proof}
Since \(A_{u+1}\subseteq A_u\), the family \((A_u)\) is nested decreasing, so the intersection is the smaller set.
\end{proof}

\begin{definition}\label{def:Lp}
Define
\[
L_p(x):=\max\{u\in\{0,1,\dots,n\}:x\in A_u\}
\qquad(x\in G_{n,p}).
\]
For \(p=1\), write \(L:=L_1\).
\end{definition}

\begin{lemma}\label{lem:Au-Lp-equiv}
Let \(1\le u\le n\) and \(x=(s,\sigma)\in G_{n,p}\). The following are equivalent:
\begin{enumerate}
\item[(i)] \(x\in A_u\).
\item[(ii)] \(\sigma^{-1}(n-u+1)<\cdots<\sigma^{-1}(n)\) and
\[
s_{\sigma^{-1}(n-u+1)}=\cdots=s_{\sigma^{-1}(n)}=0.
\]
\item[(iii)] \(L_p(x)\ge u\).
\end{enumerate}
Hence
\[
\{x:L_p(x)\ge u\}=A_u
\qquad(0\le u\le n).
\]
Also, for \(0\le u\le n-1\),
\[
L_p(x)=u
\iff
x\in A_u\setminus A_{u+1}.
\]
\end{lemma}
\begin{proof}
The equivalence \((i)\Leftrightarrow(ii)\) is Definition~\ref{def:Au}. By Definition~\ref{def:Lp},
\[
L_p(x)\ge u \iff x\in A_u,
\]
which gives \((i)\Leftrightarrow(iii)\). Since \(A_{u+1}\subseteq A_u\),
\[
L_p(x)=u \iff x\in A_u\setminus A_{u+1}. \qedhere
\]
\end{proof}

For \(p=1\),
\[
A_u=\{\pi\in S_n:\pi^{-1}(n-u+1)<\cdots<\pi^{-1}(n)\},
\qquad
|A_u|=\frac{n!}{u!}.
\]

\begin{corollary}\label{cor:Au-L}
Let \(p=1\), so \(G_{n,1}=S_n\), and write \(L:=L_1\). Then for \(1\le u\le n\),
\[
\pi\in A_u
\iff
\pi^{-1}(n-u+1)<\cdots<\pi^{-1}(n)
\iff
L(\pi)\ge u.
\]
Also,
\[
L(\pi)=u
\iff
\pi\in A_u\setminus A_{u+1}
\qquad(0\le u\le n-1).
\]
For \(1\le u\le n-1\), this is equivalently
\[
L(\pi)=u
\iff
\pi^{-1}(n-u+1)<\cdots<\pi^{-1}(n)
\ \text{and}\
\pi^{-1}(n-u)>\pi^{-1}(n-u+1).
\]
\end{corollary}

By \eqref{eq:Qbarsupp},
\[
\widetilde Q_{n-u,1}(\pi)=\frac{1}{|A_u|}\mathbf 1_{A_u}(\pi)
=\frac{u!}{n!}\mathbf 1_{\{L(\pi)\ge u\}},
\]
so \eqref{eq:mixture-u-exact} is the formula \cite[(2.5)]{DFP}.

\begin{corollary}[Distribution of \(L_p\) under \(Q^{*k}\)]\label{prop:Lp-under-Qk}
For \(0\le \ell\le n\),
\[
Q^{*k}\{L_p\ge \ell\}
=
\sum_{u=0}^{\ell-1}P_k(u)\,\frac{u!}{\ell!}\,p^{u-\ell}
+\sum_{u=\ell}^{n}P_k(u).
\]
For \(p=1\), this is \cite[(2.7)]{DFP}.
\end{corollary}

\begin{proof}
By Lemma~\ref{lem:Au-Lp-equiv}, \eqref{eq:mixture-u-exact}, and
Lemma~\ref{lem:nested-intersection},
\[
Q^{*k}\{L_p\ge \ell\}
=
\sum_{u=0}^{\ell-1}P_k(u)\,\frac{|A_\ell|}{|A_u|}
+\sum_{u=\ell}^{n}P_k(u).
\]
Since \(|A_j|=\frac{n!}{j!}p^{\,n-j}\), we get
\[
\frac{|A_\ell|}{|A_u|}=\frac{u!}{\ell!}p^{u-\ell},
\]
hence the claim follows.
\end{proof}

\subsection{One-statistic likelihood ratios}

This is the colored analog of \cite[(2.5)]{DFP}.

\begin{lemma}\label{lem:LR-one-statistic}
For every integer \(k\ge0\) and every \(x\in G_{n,p}\),
\begin{equation}\label{eq:LR-exact}
\frac{Q^{*k}(x)}{U(x)}
=\sum_{u=0}^{L_p(x)}P_k(u)\,u!\,p^u.
\end{equation}
\end{lemma}

\begin{proof}
Because $\widetilde Q_{n-u,p}$ is uniform on $A_u$, for $x\in A_u$ we have
\[
\widetilde Q_{n-u,p}(x)=\frac{1}{|A_u|}=\frac{1}{|B_{n-u}|}.
\]
Thus, by Proposition~\ref{lem:Ba-size-exact} with $a=n-u$,
\begin{equation}\label{eq:component-ratio-exact}
\frac{\widetilde Q_{n-u,p}(x)}{U(x)}
=
\frac{1/\bigl(\binom{n}{u}p^{n-u}(n-u)!\bigr)}{1/(p^n n!)}
=
u!\,p^u.
\end{equation}
If \(x\notin A_u\), the ratio is \(0\). Hence
\begin{equation}\label{eq:component-ratio-indicator}
\frac{\widetilde Q_{n-u,p}(x)}{U(x)}
=u!\,p^u\,\mathbf 1_{A_u}(x)
=u!\,p^u\,\mathbf 1_{\{L_p(x)\ge u\}}.
\end{equation}
Dividing \eqref{eq:mixture-u-exact} by \(U(x)\) gives \eqref{eq:LR-exact}.
\end{proof}

Let \(\mu\) be any probability mass function on \(\{0,1,\dots,n\}\).

\begin{definition}\label{def:LR-general-mu}
Define
\[
\widetilde Q_{\mu,p}:=\sum_{a=0}^{n}\mu(a)\,\widetilde Q_{a,p}
=\sum_{u=0}^{n}\mu(n-u)\,\widetilde Q_{n-u,p}.
\]
In particular, \(\widetilde Q_{\mu_k,p}=Q^{*k}\).
\end{definition}

\begin{remark}\label{rem:LR-general-mu}
By \eqref{eq:component-ratio-indicator},
\[
\frac{\widetilde Q_{\mu,p}(x)}{U(x)}
=\sum_{u=0}^{n}\mu(n-u)\frac{\widetilde Q_{n-u,p}(x)}{U(x)}
=\sum_{u=0}^{n}\mu(n-u)\,u!\,p^u\,\mathbf 1_{\{L_p(x)\ge u\}}
=\sum_{u=0}^{L_p(x)}\mu(n-u)\,u!\,p^u.
\]
Thus \(Q^{*k}\) is the special case \(\mu=\mu_k\).
\end{remark}

Define
\begin{equation}\label{eq:def-Smu}
S_\mu(\ell):=\sum_{u=0}^{\ell}\mu(n-u)\,u!\,p^u
\qquad(0\le \ell\le n).
\end{equation}
Then
\begin{equation}\label{eq:Smu}
\frac{\widetilde Q_{\mu,p}(x)}{U(x)}=S_\mu(L_p(x)).
\end{equation}

\begin{lemma}\label{lem:Smu-monotone}
For \(1\le \ell\le n\),
\[
S_\mu(\ell)-S_\mu(\ell-1)=\mu(n-\ell)\,\ell!\,p^\ell\ge0.
\]
Hence \(S_\mu\) is nondecreasing.
\end{lemma}

\begin{proof}
For \(1\le \ell\le n\),
\[
S_\mu(\ell)-S_\mu(\ell-1)=\mu(n-\ell)\,\ell!\,p^\ell\ge0. \qedhere
\]
\end{proof}

Set
\[
r_\mu:=\max\{u:\mu(n-u)>0\}.
\]
Equivalently,
\[
\mu(n-r_\mu)>0,
\qquad
\mu(n-u)=0\quad(u>r_\mu).
\]
That is, $r_\mu$ is the largest index $u$ for which the term $\mu(n-u)\widetilde Q_{n-u,p}$ appears with nonzero coefficient.

\begin{lemma}\label{lem:Smu-plateau-max}
We have
\[
S_\mu(\ell)=S_\mu(r_\mu)\qquad(\ell\ge r_\mu),
\]
and
\[
S_\mu(\ell)<S_\mu(r_\mu)\qquad(0\le \ell<r_\mu).
\]
Consequently,
\[
\max_{x\in G_{n,p}}\frac{\widetilde Q_{\mu,p}(x)}{U(x)}=S_\mu(r_\mu),
\]
and the maximum is attained exactly on
\[
\{x:L_p(x)\ge r_\mu\}=A_{r_\mu}.
\]
\end{lemma}

\begin{proof}
If \(\ell\ge r_\mu\), then \(\mu(n-u)=0\) for \(u>r_\mu\), so \(S_\mu(\ell)=S_\mu(r_\mu)\). If \(0\le \ell<r_\mu\), then
\[
S_\mu(r_\mu)-S_\mu(\ell)
=\sum_{u=\ell+1}^{r_\mu}\mu(n-u)\,u!\,p^u>0.
\]
By \eqref{eq:Smu}, the ratio is maximal exactly when
\[
L_p(x)\ge r_\mu.
\]
By Lemma~\ref{lem:Au-Lp-equiv},
\[
\{x\in G_{n,p}:L_p(x)\ge r_\mu\}=A_{r_\mu}.\qedhere
\]
\end{proof}

For \(p=1\), \eqref{eq:LR-exact} becomes
\[
Q^{*k}(\pi)=\frac{1}{n!}\sum_{u=0}^{L(\pi)}P_k(u)\,u!,
\]
which is the likelihood-ratio formula \cite[(2.5)]{DFP}.

\subsection{Poisson limit regime}

\begin{lemma}\label{lem:Lp-under-U}
For \(0\le u\le n\),
\begin{equation}\label{eq:Lpu}
U\{L_p\ge u\}=\frac{1}{u!\,p^u}.
\end{equation}
Hence, for \(0\le \ell\le n-1\),
\begin{equation}\label{eq:Lp-mass}
U\{L_p=\ell\}
=\frac{1}{\ell!\,p^\ell}-\frac{1}{(\ell+1)!\,p^{\ell+1}},
\end{equation}
and
\[
U\{L_p=n\}=\frac{1}{n!\,p^n}.
\]
\end{lemma}

\begin{proof}
Since \(U\) is uniform and \(\{L_p\ge u\}=A_u=\supp(B_{n-u})\),
\[
    U\{L_p\ge u\}=U(A_u)=\frac{|A_u|}{|G_{n,p}|}
=\frac{\binom{n}{u}p^{\,n-u}(n-u)!}{p^n n!}
=\frac{n!}{u!(n-u)!}\cdot\frac{p^{\,n-u}(n-u)!}{p^n n!}
=\frac{1}{u!\,p^u}.
\]
The point masses follow by subtraction.
Also,
\[
U\{L_p=n\}=U\{L_p\ge n\}=\frac{1}{n!\,p^n},
\qquad\text{since }U\{L_p\ge n+1\}=0.\qedhere
\]
\end{proof}

\begin{remark}
Equivalently, under \(U\), the event \(\{L_p\ge u\}\) is the conjunction of:
(i) the labels \(n-u+1,\dots,n\) appearing in increasing order, probability \(1/u!\); and
(ii) their colors being \(0\), probability \(p^{-u}\).
Thus
\begin{equation}\label{eq:Lp-tail}
U\{L_p\ge u\}=\frac{1}{u!\,p^u}.
\end{equation}
\end{remark}

\begin{corollary}\label{cor:Lp-law-general-mu}
For \(0\le \ell\le n-1\),
\[
\widetilde Q_{\mu,p}\{L_p=\ell\}
=
\left(\frac{1}{\ell!\,p^\ell}-\frac{1}{(\ell+1)!\,p^{\ell+1}}\right)S_\mu(\ell),
\]
and
\[
\widetilde Q_{\mu,p}\{L_p=n\}
=
\frac{1}{n!\,p^n}S_\mu(n).
\]
\end{corollary}

\begin{proof}
By \eqref{eq:Smu}, the ratio \(\widetilde Q_{\mu,p}(x)/U(x)\) is constant on \(\{L_p=\ell\}\), equal to \(S_\mu(\ell)\). Thus
\[
\widetilde Q_{\mu,p}\{L_p=\ell\}
=
\sum_{x:\,L_p(x)=\ell}U(x)\,S_\mu(\ell)
=
U\{L_p=\ell\}\,S_\mu(\ell),
\]
and Lemma~\ref{lem:Lp-under-U} finishes the proof.
\end{proof}

\begin{proposition}\label{prop:TV-likelihood}
If \(M\ll U\), then
\begin{equation}\label{eq:MUTV}
\|M-U\|_{\TV}
=\mathbb E_U\Bigl[\Bigl(\frac{M(X)}{U(X)}-1\Bigr)_+\Bigr].
\end{equation}
\end{proposition}

\begin{proof}
Let \(Z:=M(X)/U(X)-1\). Then
\begin{align*}
\|M-U\|_{\TV}
&=\frac12\sum_x |M(x)-U(x)|
=\frac12\sum_x U(x)\left|\frac{M(x)}{U(x)}-1\right|
=\frac12\,\mathbb E_U\left[\left|\frac{M(X)}{U(X)}-1\right|\right]
=\frac12\,\mathbb E_U[|Z|].
\end{align*}
Also,
\[
\mathbb E_U[Z]
=\mathbb E_U\!\left[\frac{M(X)}{U(X)}\right]-1
=\sum_x U(x)\frac{M(x)}{U(x)}-1
=\sum_x M(x)-1
=0,
\]
so \(\mathbb E_U[Z_+]=\mathbb E_U[(-Z)_+]\). Hence
\[
\|M-U\|_{\TV}
=\frac12\bigl(\mathbb E_U[Z_+]+\mathbb E_U[(-Z)_+]\bigr)
=\mathbb E_U[Z_+]
=\mathbb E_U\Bigl[\Bigl(\frac{M(X)}{U(X)}-1\Bigr)_+\Bigr]. \qedhere
\]
\end{proof}

\begin{theorem}[Exact TV formula for nested-set mixtures]\label{prop:TV-exact-general-mu}
If \(\widetilde Q_{\mu,p}=U\), then
\[
\|\widetilde Q_{\mu,p}-U\|_{\TV}=0.
\]
Assume \(\widetilde Q_{\mu,p}\neq U\), and define
\[
w_\mu:=\min\{\ell\in\{0,1,\dots,n\}: S_\mu(\ell)>1\}.
\]
Then \(w_\mu\in\{1,\dots,n\}\) (indeed \(w_\mu\ge2\) when \(p=1\)), and
\[
\|\widetilde Q_{\mu,p}-U\|_{\TV}
=\widetilde Q_{\mu,p}(A_{w_\mu})-U(A_{w_\mu}).
\]
Moreover,
\[
\|\widetilde Q_{\mu,p}-U\|_{\TV}
=
1-\sum_{u=0}^{w_\mu-1}\mu(n-u)
+\frac{1}{w_\mu!\,p^{w_\mu}}
\left(\sum_{u=0}^{w_\mu-1}\mu(n-u)\,u!\,p^u-1\right).
\]
\end{theorem}

\begin{proof}
If \(\widetilde Q_{\mu,p}=U\), then
\[
\|\widetilde Q_{\mu,p}-U\|_{\TV}=0.
\]
Assume henceforth that \(\widetilde Q_{\mu,p}\neq U\).

By \eqref{eq:Smu},
\[
\frac{\widetilde Q_{\mu,p}(x)}{U(x)}=S_\mu(L_p(x))
\qquad(x\in G_{n,p}).
\]
Hence
\[
1=\sum_{x\in G_{n,p}}\widetilde Q_{\mu,p}(x)
=\sum_{x\in G_{n,p}}U(x)\,S_\mu(L_p(x))
=\sum_{\ell=0}^{n}U\{L_p=\ell\}\,S_\mu(\ell).
\]

\smallskip
\noindent\textbf{Claim 1.}
The index \(w_\mu\) is well-defined and satisfies \(w_\mu\in\{1,\dots,n\}\).
Moreover, if \(p=1\), then \(w_\mu\ge2\).

\smallskip
\noindent\emph{Proof of Claim 1.}
By Lemma~\ref{lem:Smu-monotone}, \(S_\mu\) is nondecreasing, and
\[
S_\mu(0)=\mu(n)\le 1.
\]
If \(S_\mu(\ell)\le 1\) for all \(\ell\), then
\[
1=\sum_{\ell=0}^{n}U\{L_p=\ell\}\,S_\mu(\ell)
\]
shows that the random variable \(S_\mu(L_p)\) is bounded above by \(1\) and has \(U\)-expectation \(1\). Hence
\[
S_\mu(L_p)=1
\qquad U\text{-a.s.}
\]
Since \(U(x)>0\) for every \(x\in G_{n,p}\), this means
\[
S_\mu(L_p(x))=1,
\qquad \forall x\in G_{n,p}.
\]
Therefore, by \eqref{eq:Smu},
\[
\frac{\widetilde Q_{\mu,p}(x)}{U(x)}=1
\qquad(x\in G_{n,p}),
\]
so \(\widetilde Q_{\mu,p}=U\), a contradiction. Thus \(w_\mu\) is well-defined, and \(w_\mu\ge1\).

If \(p=1\), then
\[
S_\mu(1)=\mu(n)+\mu(n-1)\le \sum_{a=0}^n \mu(a)=1,
\]
so \(w_\mu\ge2\). This proves Claim~1.

\smallskip
By Proposition~\ref{prop:TV-likelihood},
\[
\|\widetilde Q_{\mu,p}-U\|_{\TV}
=\mathbb E_U\bigl[(S_\mu(L_p)-1)_+\bigr].
\]
Since \(S_\mu\) is nondecreasing and \(w_\mu\) is the first index with
\(S_\mu(\ell)>1\), we have
\[
S_\mu(\ell)\le1 \quad (\ell<w_\mu),
\qquad
S_\mu(\ell)>1 \quad (\ell\ge w_\mu).
\]
Therefore
\[
(S_\mu(L_p)-1)_+=(S_\mu(L_p)-1)\mathbf 1_{\{L_p\ge w_\mu\}}.
\]

\smallskip
\noindent\textbf{Claim 2.}
\[
\|\widetilde Q_{\mu,p}-U\|_{\TV}
=\widetilde Q_{\mu,p}(A_{w_\mu})-U(A_{w_\mu}).
\]

\smallskip
\noindent\emph{Proof of Claim 2.}
By Lemma~\ref{lem:Au-Lp-equiv},
\[
\{L_p\ge w_\mu\}=A_{w_\mu}.
\]
Hence by Proposition~\ref{prop:TV-likelihood} and \eqref{eq:Smu},
\begin{align*}
\|\widetilde Q_{\mu,p}-U\|_{\TV}
&=
\sum_{x\in G_{n,p}}U(x)\,(S_\mu(L_p(x))-1)\mathbf 1_{A_{w_\mu}}(x)\\
&=
\sum_{x\in A_{w_\mu}}
\left(\widetilde Q_{\mu,p}(x)-U(x)\right)\\
&=
\widetilde Q_{\mu,p}(A_{w_\mu})-U(A_{w_\mu}).
\end{align*}
This proves Claim~2.

\smallskip
It remains to evaluate \(\widetilde Q_{\mu,p}(A_{w_\mu})\). By definition,
\[
\widetilde Q_{\mu,p}
=\sum_{u=0}^{n}\mu(n-u)\,\widetilde Q_{n-u,p},
\]
so
\[
\widetilde Q_{\mu,p}(A_{w_\mu})
=
\sum_{u=0}^{n}\mu(n-u)\,\widetilde Q_{n-u,p}(A_{w_\mu}).
\]

Since \(\widetilde Q_{n-u,p}\) is uniform on \(A_u\) by \eqref{eq:Qbarsupp},
\[
\widetilde Q_{n-u,p}(A_{w_\mu})
=
\frac{|A_u\cap A_{w_\mu}|}{|A_u|}.
\]
By Lemma~\ref{lem:nested-intersection}, \(A_u\cap A_{w_\mu}=A_{\max(u,w_\mu)}\).
Hence
\[
\widetilde Q_{n-u,p}(A_{w_\mu})
=
\begin{cases}
|A_{w_\mu}|/|A_u|, & u<w_\mu,\\
1, & u\ge w_\mu.
\end{cases}
\]
Also, by \eqref{eq:Lp-tail},
\[
\frac{|A_{w_\mu}|}{|A_u|}
=\frac{U(A_{w_\mu})}{U(A_u)}
=\frac{u!\,p^u}{w_\mu!\,p^{w_\mu}}
\qquad(u<w_\mu),
\]
and
\[
U(A_{w_\mu})=\frac{1}{w_\mu!\,p^{w_\mu}}.
\]
Therefore
\begin{align*}
\widetilde Q_{\mu,p}(A_{w_\mu})
&=
\sum_{u=0}^{w_\mu-1}\mu(n-u)\frac{u!\,p^u}{w_\mu!\,p^{w_\mu}}
+\sum_{u=w_\mu}^{n}\mu(n-u)\\
&=
1-\sum_{u=0}^{w_\mu-1}\mu(n-u)
+\frac{1}{w_\mu!\,p^{w_\mu}}
\sum_{u=0}^{w_\mu-1}\mu(n-u)\,u!\,p^u.
\end{align*}
Subtracting \(U(A_{w_\mu})=1/(w_\mu!\,p^{w_\mu})\) and using Claim~2 gives
\[
\|\widetilde Q_{\mu,p}-U\|_{\TV}
=
1-\sum_{u=0}^{w_\mu-1}\mu(n-u)
+\frac{1}{w_\mu!\,p^{w_\mu}}
\left(\sum_{u=0}^{w_\mu-1}\mu(n-u)\,u!\,p^u-1\right).\qedhere
\]
\end{proof}

Now assume that, for some fixed \(\lambda>0\), equivalently \(c:=-\log\lambda\in\R\),
\begin{equation}\label{eq:mu-Poisson-assump}
\mu_n(n-u)\longrightarrow e^{-\lambda}\frac{\lambda^u}{u!}
\qquad(n\to\infty)
\end{equation}
for each fixed \(u\ge0\). Write
\[
r:=p\lambda,
\qquad
s(\ell):=e^{-\lambda}\sum_{u=0}^{\ell}r^u
\qquad(\ell\ge0).
\]
Then \(s(\ell)\) is strictly increasing and \(s(0)=e^{-\lambda}<1\). If \(r\ge1\), then
\[
s(\ell)\to+\infty
\qquad(\ell\to\infty).
\]
If \(0<r<1\), then
\[
s(\ell)\to \frac{e^{-\lambda}}{1-r}
\qquad(\ell\to\infty).
\]
Since \(r=p\lambda\ge\lambda\), in the subcase \(0<r<1\) we have
\[
0<1-r\le 1-\lambda<e^{-\lambda},
\]
and therefore
\[
\frac{e^{-\lambda}}{1-r}>1.
\]
Hence in all cases there is a unique integer \(w^*=w^*(c,p)\ge1\) such that
\begin{equation}\label{eq:wstar}
e^{-\lambda}\sum_{u=0}^{w^*-1}r^u\le1<
e^{-\lambda}\sum_{u=0}^{w^*}r^u.
\end{equation}
Equivalently, \(w^*-1\) is the largest \(\ell\) with \(s(\ell)\le1\). When \(p=1\), this \(w^*\) equals \(\ell^*+1\) in \cite[Proposition~3.1]{DFP}.

\begin{lemma}\label{lem:wstar-closed}
With \(r=p\lambda\), the unique integer \(w^*\) satisfying \eqref{eq:wstar} is
\[
w^*=
\begin{cases}
\displaystyle \left\lfloor \frac{\log\!\bigl(1+e^{\lambda}(r-1)\bigr)}{\log r}\right\rfloor,
& r\neq 1,\\[10pt]
\displaystyle \lfloor e^{\lambda}\rfloor,
& r=1.
\end{cases}
\]
\end{lemma}

\begin{proof}
The inequality \eqref{eq:wstar} is equivalent to
\[
\sum_{u=0}^{w^*-1}r^u\le e^\lambda<\sum_{u=0}^{w^*}r^u.
\]
If \(r=1\), this reads
\[
w^*\le e^\lambda<w^*+1,
\]
so \(w^*=\lfloor e^\lambda\rfloor\).

If \(r\neq1\), use
\[
\sum_{u=0}^{m}r^u=\frac{r^{m+1}-1}{r-1}
\]
to obtain
\[
\frac{r^{w^*}-1}{r-1}\le e^\lambda<\frac{r^{w^*+1}-1}{r-1}.
\]
If \(r>1\), then
\[
r^{w^*}\le 1+e^\lambda(r-1)<r^{w^*+1}.
\]
If \(0<r<1\), then multiplying by \(r-1<0\) reverses the inequalities, so
\[
r^{w^*}\ge 1+e^\lambda(r-1)>r^{w^*+1}.
\]
In this case \(1+e^\lambda(r-1)>0\) because
\[
\frac{e^{-\lambda}}{1-r}>1.
\]
Therefore, in both cases,
\[
w^*\le \frac{\log(1+e^\lambda(r-1))}{\log r}<w^*+1.
\]
Hence
\[
w^*=\left\lfloor \frac{\log(1+e^\lambda(r-1))}{\log r}\right\rfloor.\qedhere
\]
\end{proof}

\begin{theorem}[General Poisson-regime TV limit]\label{prop:DFP3.1-colored}
Let \(w^*=w^*(c,p)\) be given by \eqref{eq:wstar}. Then
\[
\|M_n-U\|_{\TV}\longrightarrow f_p(c),
\]
where
\begin{equation}\label{eq:fp-explicit0}
f_p(c)=
1-e^{-\lambda}\sum_{u=0}^{w^*-1}\frac{\lambda^u}{u!}
+\frac{1}{w^*!\,p^{w^*}}
\left(e^{-\lambda}\sum_{u=0}^{w^*-1}r^u-1\right).
\end{equation}
\end{theorem}

\begin{proof}
Fix \(\ell\ge0\). By \eqref{eq:mu-Poisson-assump},
\begin{equation}\label{eq:Smu-pointwise-limit}
S_{\mu_n}(\ell)=\sum_{u=0}^{\ell}\mu_n(n-u)\,u!\,p^u
\longrightarrow
e^{-\lambda}\sum_{u=0}^{\ell}(p\lambda)^u
=s(\ell).
\end{equation}

For all sufficiently large \(n\), define
\[
w_n:=w_{\mu_n}.
\]
Since \(S_{\mu_n}(w^*)\to s(w^*)>1\), the measures \(\widetilde Q_{\mu_n,p}\) are not equal to \(U\) for all sufficiently large \(n\), so \(w_n\) is well-defined by Theorem~\ref{prop:TV-exact-general-mu}.

Also,
\[
S_{\mu_n}(w^*-1)\to s(w^*-1)\le1,
\qquad
S_{\mu_n}(w^*)\to s(w^*)>1.
\]
Hence, for all sufficiently large \(n\),
\[
S_{\mu_n}(w^*)>1.
\]
Now:
\begin{itemize}
\item If \(s(w^*-1)<1\), then also
\[
S_{\mu_n}(w^*-1)<1
\]
for all sufficiently large \(n\). Therefore the first index \(\ell\) for which \(S_{\mu_n}(\ell)>1\) is exactly \(w^*\), so
\[
w_n=w^*.
\]
\item If \(s(w^*-1)=1\), then \(w^*\ge 2\) and, by strict increase of \(s\),
\[
s(w^*-2)<1.
\]
Thus, for all sufficiently large \(n\),
\[
S_{\mu_n}(w^*-2)<1
\qquad\text{and}\qquad
S_{\mu_n}(w^*)>1.
\]
So the first index \(\ell\) for which \(S_{\mu_n}(\ell)>1\) can only be \(w^*-1\) or \(w^*\), meaning
\[
w_n\in\{w^*-1,w^*\}.
\]
\end{itemize}

By Theorem~\ref{prop:TV-exact-general-mu},
\[
\|\widetilde Q_{\mu_n,p}-U\|_{\TV}
=
1-\sum_{u=0}^{w_n-1}\mu_n(n-u)
+\frac{1}{w_n!\,p^{w_n}}
\left(\sum_{u=0}^{w_n-1}\mu_n(n-u)\,u!\,p^u-1\right).
\]

\begin{itemize}
\item If \(s(w^*-1)<1\), then \(w_n=w^*\) eventually, so termwise convergence gives
\[
\|\widetilde Q_{\mu_n,p}-U\|_{\TV}\to f_p(c).
\]
\item If \(s(w^*-1)=1\), then \(w_n\in\{w^*-1,w^*\}\) eventually. In the case \(w_n=w^*-1\), Theorem~\ref{prop:TV-exact-general-mu} also gives
\[
\|\widetilde Q_{\mu_n,p}-U\|_{\TV}
=\widetilde Q_{\mu_n,p}(A_{w^*-1})-U(A_{w^*-1}).
\]
\end{itemize}
In the tie case \(s(w^*-1)=1\), the discrepancy
\begin{align*}
&\Bigl(\widetilde Q_{\mu_n,p}(A_{w^*-1})-U(A_{w^*-1})\Bigr)
-\Bigl(\widetilde Q_{\mu_n,p}(A_{w^*})-U(A_{w^*})\Bigr)\\
&\qquad=
\widetilde Q_{\mu_n,p}\{L_p=w^*-1\}-U\{L_p=w^*-1\}\\
&\qquad=
\bigl(S_{\mu_n}(w^*-1)-1\bigr)\,U\{L_p=w^*-1\}=o(1),
\end{align*}
since \(S_{\mu_n}(w^*-1)\to1\). Therefore the same limit \(f_p(c)\) holds in both subcases.
\end{proof}

\begin{remark}\label{rem:fp-extend-sums}
Equivalently, we can also write \eqref{eq:fp-explicit0} with the upper limits extended to $w^*$:
\[
f_p(c)
=
1-e^{-\lambda}\sum_{u=0}^{w^*}\frac{\lambda^u}{u!}
+\frac{1}{w^*!\,p^{w^*}}
\left(e^{-\lambda}\sum_{u=0}^{w^*}r^u-1\right).
\]
Indeed,
\[
e^{-\lambda}\frac{\lambda^{w^*}}{w^*!}
=\frac{1}{w^*!\,p^{w^*}}\,e^{-\lambda}(p\lambda)^{w^*}
=\frac{1}{w^*!\,p^{w^*}}\,e^{-\lambda}r^{w^*}.
\]
\end{remark}

In particular:

\textbf{Case 1.} \(c\ge0,\ p\ge2\).
Then \(0<\lambda\le1\), so \(e^{-\lambda}<1\). Also,
\[
e^\lambda
=1+\lambda+\sum_{j=2}^{\infty}\frac{\lambda^j}{j!}
\le 1+\lambda+\lambda\sum_{j=2}^{\infty}\frac1{j!}
<1+2\lambda
\le 1+p\lambda.
\]
Hence
\[
1<e^{-\lambda}(1+p\lambda),
\]
so \(w^*=1\) by \eqref{eq:wstar}. Therefore
\[
f_p(c)=1-e^{-\lambda}+\frac1p(e^{-\lambda}-1)
=\Bigl(1-\frac1p\Bigr)\bigl(1-e^{-\lambda}\bigr)
=\Bigl(1-\frac1p\Bigr)\Bigl(1-e^{-e^{-c}}\Bigr).
\]

\textbf{Case 2.} \(c\ge0,\ p=1\).
Then \(0<\lambda\le1\). Since \(e^\lambda\ge 1+\lambda\),
\[
e^{-\lambda}(1+\lambda)\le1.
\]
Also,
\[
e^\lambda
=1+\lambda+\frac{\lambda^2}{2}+\sum_{j=3}^{\infty}\frac{\lambda^j}{j!}
\le 1+\lambda+\frac{\lambda^2}{2}+\lambda^2\sum_{j=3}^{\infty}\frac1{j!}
<1+\lambda+\lambda^2.
\]
Thus
\[
1<e^{-\lambda}(1+\lambda+\lambda^2),
\]
so \(w^*=2\). Therefore
\[
f_1(c)=1-e^{-\lambda}(1+\lambda)+\frac12\bigl(e^{-\lambda}(1+\lambda)-1\bigr)
=\frac12\Bigl(1-e^{-\lambda}(1+\lambda)\Bigr),
\]
which is the \(c\ge0\) specialization of \cite[Theorem~1.1]{DFP}.

\textbf{Case 3.} \(c<0\).
Then \(\lambda>1\) and \(r=p\lambda>1\), so Lemma~\ref{lem:wstar-closed} gives
\[
w^*=\left\lfloor \frac{\log\!\bigl(1+e^\lambda(r-1)\bigr)}{\log r}\right\rfloor.
\]
Substituting this into \eqref{eq:fp-explicit0} gives the profile. For \(p=1\), writing \(\ell^*:=w^*-1\) recovers \cite[(1.5b)]{DFP}.

\subsection{Asymptotics as \texorpdfstring{$c\to\pm\infty$}{c to plus/minus infinity}}

We explore the two extreme regimes: \(c\to+\infty\iff \lambda\to0\); and \(c\to-\infty\iff \lambda\to\infty\).

If \(p\ge2\), then \(w^*(c,p)=1\) for all \(c\ge0\), so
\[
f_p(c)=\Bigl(1-\frac1p\Bigr)\bigl(1-e^{-\lambda}\bigr)
=\Bigl(1-\frac1p\Bigr)e^{-c}+O(e^{-2c}).
\]
If \(p=1\), then \(w^*(c,1)=2\) for all \(c\ge0\), so
\[
f_1(c)=\frac12\Bigl(1-e^{-\lambda}(1+\lambda)\Bigr)
=\frac14 e^{-2c}+O(e^{-3c}),
\]
agreeing with \cite[Remarks~2]{DFP}. Hence \(f_p(c)\to0\) as \(c\to+\infty\) in all cases.

Now let \(c\to-\infty\), equivalently \(\lambda\to\infty\), and set
\[
p_\lambda(u):=e^{-\lambda}\frac{\lambda^u}{u!},
\qquad
F_\lambda(m):=\sum_{u=0}^{m}p_\lambda(u).
\]
Let \(w^*=w^*(c,p)\) be given by \eqref{eq:wstar}. By Remark~\ref{rem:fp-extend-sums},
\begin{equation}\label{eq:1-minus-fp-decomp2}
1-f_p(c)=F_\lambda(w^*)-\Delta_\lambda,
\qquad
\Delta_\lambda:=\frac{1}{w^*!\,p^{w^*}}
\Bigl(e^{-\lambda}\sum_{u=0}^{w^*}r^u-1\Bigr).
\end{equation}
Since \(e^{-\lambda}\sum_{u=0}^{w^*}r^u>1\) by \eqref{eq:wstar}, we have \(\Delta_\lambda>0\). Also,
\[
e^{-\lambda}\sum_{u=0}^{w^*}r^u-1
=e^{-\lambda}r^{w^*}
+\Bigl(e^{-\lambda}\sum_{u=0}^{w^*-1}r^u-1\Bigr)
\le e^{-\lambda}r^{w^*},
\]
so
\begin{equation}\label{eq:Delta-upper}
0<\Delta_\lambda\le \frac{e^{-\lambda}r^{w^*}}{w^*!\,p^{w^*}}
=\frac{e^{-\lambda}(p\lambda)^{w^*}}{w^*!\,p^{w^*}}
=e^{-\lambda}\frac{\lambda^{w^*}}{w^*!}
=p_\lambda(w^*).
\end{equation}
Combining \eqref{eq:1-minus-fp-decomp2} and \eqref{eq:Delta-upper} yields the Poisson sandwich
\begin{equation}\label{eq:1-minus-fp-bounds}
F_\lambda(w^*-1)\le1-f_p(c)\le F_\lambda(w^*).
\end{equation}

\begin{lemma}\label{lem:wstar-asymp}
As \(\lambda\to\infty\),
\[
w^*=\frac{\lambda}{\log r}+O(1)
=\frac{\lambda}{\log(p\lambda)}(1+o(1)),
\qquad
\frac{w^*}{\lambda}\to0,
\]
and
\[
\log\!\Bigl(\frac{w^*}{\lambda}\Bigr)
=-\log\log(p\lambda)+o(1)=o(\lambda).
\]
\end{lemma}

\begin{proof}
By Lemma~\ref{lem:wstar-closed},
\[
w^*=\left\lfloor \frac{\log(1+e^\lambda(r-1))}{\log r}\right\rfloor.
\]
Since \(r=p\lambda\to\infty\), we have \(\log(r-1)=\log r+o(1)\). Thus
\[\log(1+e^\lambda(r-1))
=
\lambda+\log(r-1)+\log\!\left(1+\frac{1}{e^\lambda(r-1)}\right)
=
\lambda+\log r+o(1),\]
which gives the first asymptotic. Hence
\[
w^*=\frac{\lambda}{\log r}+O(1)=\frac{\lambda}{\log(p\lambda)}(1+o(1)),
\]
so \(w^*/\lambda\to0\).

Thus
\[
\frac{w^*}{\lambda}=\frac{1}{\log(p\lambda)}(1+o(1)),
\]
and taking \(\log\) gives
\[
\log\!\Bigl(\frac{w^*}{\lambda}\Bigr)
=-\log\log(p\lambda)+o(1).
\]
In particular, since \(\log\log(p\lambda)=o(\lambda)\) as \(\lambda\to\infty\), we have
\[
\log\!\Bigl(\frac{w^*}{\lambda}\Bigr)=o(\lambda).\qedhere
\]
\end{proof}

\begin{proposition}[Poisson-tail estimate {\cite[Remark~3]{DFP}}]\label{prop:Poisson-tail-DFP}
Let \(t=t(\lambda)\in\mathbb Z_{\ge0}\) satisfy \(t/\lambda=o(1)\). Then
\[
1\le \frac{F_\lambda(t)}{p_\lambda(t)}
\le \sum_{j=0}^{t}\Bigl(\frac{t}{\lambda}\Bigr)^j
\le \frac{1}{1-t/\lambda}
=1+o(1).
\]
In particular,
\[
F_\lambda(t)=(1+o(1))\,p_\lambda(t).
\]
\end{proposition}

\begin{proof}
For \(0\le j\le t\),
\[
\frac{p_\lambda(t-j)}{p_\lambda(t)}
=
\frac{e^{-\lambda}\lambda^{t-j}/(t-j)!}{e^{-\lambda}\lambda^t/t!}
=
\frac{(t)_j}{\lambda^j}
\le
\left(\frac{t}{\lambda}\right)^j.
\]
Therefore
\[
\frac{F_\lambda(t)}{p_\lambda(t)}
=
\sum_{j=0}^{t}\frac{p_\lambda(t-j)}{p_\lambda(t)}
=
\sum_{j=0}^{t}\frac{(t)_j}{\lambda^j}
\le
\sum_{j=0}^{t}\left(\frac{t}{\lambda}\right)^j.
\]
The lower bound
\[
1\le \frac{F_\lambda(t)}{p_\lambda(t)}
\]
is immediate from the \(j=0\) term. Since \(t/\lambda=o(1)\), we have \(t/\lambda<1\) for all sufficiently large \(\lambda\), and hence
\[
\sum_{j=0}^{t}\left(\frac{t}{\lambda}\right)^j
\le \frac{1}{1-t/\lambda}
=1+o(1).
\]
This proves the claim.
\end{proof}

Applying Proposition~\ref{prop:Poisson-tail-DFP} with \(t=w^*\) and \(t=w^*-1\) gives
\begin{equation}\label{eq:F-asymp-p}
F_\lambda(w^*)=(1+o(1))\,p_\lambda(w^*),
\qquad
F_\lambda(w^*-1)=(1+o(1))\,p_\lambda(w^*-1).
\end{equation}

\begin{lemma}\label{lem:poisson-pmf-wstar-scale}
As \(\lambda\to\infty\),
\begin{equation}\label{eq:p-wstar-exp-scale}
p_\lambda(w^*)=\exp\bigl(-\lambda+o(\lambda)\bigr)
=\exp\bigl(-(1+o(1))\lambda\bigr).
\end{equation}
\end{lemma}

\begin{proof}
Since
\[
p_\lambda(w^*)=e^{-\lambda}\frac{\lambda^{w^*}}{w^*!},
\]
we have
\[
\log p_\lambda(w^*)=-\lambda+w^*\log\lambda-\log(w^*!).
\]
By Stirling,
\[
\log(w^*!)=w^*\log w^*-w^*+O(\log w^*),
\]
so
\[
\log p_\lambda(w^*)
=-\lambda+w^*\Bigl(1+\log\frac{\lambda}{w^*}\Bigr)+O(\log w^*).
\]
Now Lemma~\ref{lem:wstar-asymp} gives
\[
w^*=\frac{\lambda}{\log r}(1+o(1)),
\qquad
\log\frac{\lambda}{w^*}=\log\log r+o(1),
\qquad
O(\log w^*)=o(w^*)=o(\lambda).
\]
Therefore
\[
\log p_\lambda(w^*)
=-\lambda+\frac{\lambda}{\log r}\bigl(1+\log\log r\bigr)+o(\lambda)
=-\lambda+o(\lambda),
\]
since \((1+\log\log r)/\log r\to0\).
Exponentiating gives \eqref{eq:p-wstar-exp-scale}.
\end{proof}

\begin{theorem}[Doubly exponential approach to one]\label{thm:doubly-exp}
Fix \(p\ge1\). As \(c\to-\infty\), equivalently \(\lambda\to\infty\),
\[
1-f_p(c)=\exp\bigl(-\lambda+o(\lambda)\bigr)
=\exp\bigl(-(1+o(1))\lambda\bigr)
=\exp\bigl(-(1+o(1))e^{-c}\bigr).
\]
In particular, \(f_p(c)\to1\) doubly exponentially fast as \(c\to-\infty\).
\end{theorem}

\begin{proof}
By \eqref{eq:1-minus-fp-bounds} and \eqref{eq:F-asymp-p},
\[
(1+o(1))\,p_\lambda(w^*-1)\le1-f_p(c)\le(1+o(1))\,p_\lambda(w^*).
\]
Taking $\log$ throughout,
\[
\log p_\lambda(w^*-1)+o(1)
 \le \log(1-f_p(c))
 \le \log p_\lambda(w^*)+o(1).
\]
Since
\[
p_\lambda(w^*-1)=p_\lambda(w^*)\frac{w^*}{\lambda},
\]
we get
\[
\log p_\lambda(w^*-1)
=\log p_\lambda(w^*)+\log\!\Bigl(\frac{w^*}{\lambda}\Bigr)
=\log p_\lambda(w^*)+o(\lambda)
\]
by Lemma~\ref{lem:wstar-asymp}. Hence the sandwich implies
\[
\log(1-f_p(c))=\log p_\lambda(w^*)+o(\lambda).
\]
Finally, by Lemma \ref{lem:poisson-pmf-wstar-scale}, $\log p_\lambda(w^*)=-\lambda+o(\lambda)$, hence
\[
\log(1-f_p(c))=-(1+o(1))\lambda,
\qquad\text{i.e.}\qquad
1-f_p(c)=\exp\bigl(-(1+o(1))\lambda\bigr). \qedhere
\]
\end{proof}

\begin{remark}\label{rem:fp-endpoints}
The endpoint limits for \(f_p\) were established above, and its monotonicity is proved later in Corollary~\ref{cor:profile-monotone}. Thus \(f_p\) has the shape of a cutoff profile:
\[
f_p(c)\to1 \quad(c\to-\infty),
\qquad
f_p(c)\to0 \quad(c\to+\infty).
\]
\end{remark}

\section{Colored top-\texorpdfstring{$m$}{m}-to-random}\label{sec:top-m}

This section specializes the general nested-set mixture framework to colored top-\(m\)-to-random. The \(k\)-step law is identified with an \(m\)-subset coupon-collector mixture, and the empty-label Poisson limit then gives the total-variation cutoff profile.

Fix \(1\le m\le n\). Recall
\[
Q^{(m)}:=\widetilde Q_{m,p}=\frac{1}{|B_m|}B_m.
\]
By Proposition~\ref{lem:Ba-size-exact},
\[
|B_m|=\binom{n}{m}p^m m!=(n)_m\,p^m,
\qquad
(n)_m:=n(n-1)\cdots(n-m+1).
\]
Thus \(Q^{(m)}\) is the uniform law on \(\supp(B_m)\). Also,
\[
Q^{(1)}=Q,
\qquad
Q^{(n)}=\widetilde Q_{n,p}=U,
\qquad
(Q^{(m)})^{*k}=(Q^{(m)})^k
\]
(product in \(\mathbb RG_{n,p}\)).

\subsection{Exact mixture}

We again expand the \(k\)-step law over the nested supports
\[
A_u=\supp(B_{n-u})=\supp(\widetilde Q_{n-u,p}),
\]
but now the weights come from an \(m\)-sampling occupancy model.

\begin{definition}\label{def:occupancy-m}
Fix integers \(n\ge1\), \(1\le m\le n\), and \(k\ge0\). For each round \(t=1,\dots,k\), choose a subset
\[
S_t\subseteq[n],\qquad |S_t|=m,
\]
uniformly from \(\binom{[n]}{m}\), independently across rounds. Let
\[
A:=\left|\bigcup_{t=1}^k S_t\right|,
\qquad
E:=n-A.
\]
Thus \(A\) is the number of indices ever chosen, and \(E\) is the number never chosen. Define
\[
\mu_k^{(m)}(a):=\Prob(A=a),
\qquad
P_k^{(m)}(u):=\Prob(E=u)
\qquad(0\le a,u\le n).
\]
When \(m=1\), this is the ordinary balls-into-boxes model from Section~\ref{sec:mixture-engine}; in particular,
\[
\mu_k^{(1)}(a)=\mu_k(a),
\qquad
P_k^{(1)}(u)=P_k(u).
\]
\end{definition}

\begin{example}
Let $n=5$, $m=2$, and $k=3$. In each round choose a uniform $2$-subset $S_t\subseteq[n]$. Suppose
\[
S_1=\{1,4\},\qquad S_2=\{4,5\},\qquad S_3=\{2,4\}.
\]
Then $\bigcup_{t=1}^3 S_t=\{1,2,4,5\}$, so $A=\left|\bigcup_{t=1}^3 S_t\right|=4$ and $E=n-A=1$ (the only never-chosen label is $3$).
\end{example}

\begin{lemma}[Exact top-\texorpdfstring{$m$}{m} occupancy mixture]\label{prop:Qm-mixture}
Fix \(1\le m\le n\) and \(k\ge0\). Then
\begin{align}
(Q^{(m)})^{*k}
&=\sum_{a=0}^{n}\mu_k^{(m)}(a)\,\widetilde Q_{a,p},
\label{eq:Qm-mixture-a}\\
&=\sum_{u=0}^{n}P_k^{(m)}(u)\,\widetilde Q_{n-u,p}.
\label{eq:Qm-mixture-u}
\end{align}
In particular,
\[
P_k^{(m)}(u)=\mu_k^{(m)}(n-u),
\qquad
(Q^{(m)})^{*k}=\widetilde Q_{\mu_k^{(m)},p}.
\]
If \(k\ge1\), then
\begin{equation}\label{eq:mu-m-explicit}
\mu_k^{(m)}(a)
=
\frac{1}{(n)_m^k}\binom{n}{a}
\sum_{j=0}^{a}(-1)^j\binom{a}{j}(a-j)_m^k,
\qquad
0\le a\le n,
\end{equation}
with the convention \((x)_m=0\) for \(x<m\). 
\end{lemma}

\begin{proof}
If \(k=0\), then
\[
(Q^{(m)})^{*0}=\delta_e=\widetilde Q_{0,p},
\qquad
\mu_0^{(m)}(0)=1,
\qquad
P_0^{(m)}(n)=1,
\]
so \eqref{eq:Qm-mixture-a} and \eqref{eq:Qm-mixture-u} are immediate. Assume henceforth that \(k\ge1\).

By \(|B_m|=(n)_m p^m\),
\begin{equation}\label{eq:Qm-from-Bm}
(Q^{(m)})^{*k}=(Q^{(m)})^k=\frac{1}{|B_m|^k}\,B_m^k
=\frac{1}{(n)_m^k p^{mk}}\,B_m^k.
\end{equation}
By \cite[Theorem~2.7]{NSS},
\begin{equation}\label{eq:Bm-stirling}
\begin{aligned}
B_m^k
&=\sum_{a=m}^{n} p^{mk-a}\,S_m(k,a)\,B_a\\
&=\sum_{a=0}^{n} p^{mk-a}\,S_m(k,a)\,B_a
\qquad\text{since }S_m(k,a)=0\text{ for }0\le a\le m-1,
\end{aligned}
\end{equation}
where \(S_m(k,a)\) is the generalized Stirling number \(S_{m,m}(k,a)\)
in the notation of \cite{NSS}, given by
\begin{equation}\label{eq:Sm-explicit}
S_m(k,a)
=\frac{1}{a!}\sum_{j=0}^{a}(-1)^j\binom{a}{j}(a-j)_m^k,
\end{equation}
and \(S_m(k,a)=0\) for \(a<m\). For \(m=1\), this reduces to the ordinary Stirling formula
\[
S_1(k,a)=\left\{\!\!\begin{matrix}k\\ a\end{matrix}\!\!\right\}.
\]

Substituting \eqref{eq:Bm-stirling} into \eqref{eq:Qm-from-Bm} and using
\[
B_a=|B_a|\widetilde Q_{a,p},
\qquad
|B_a|=\binom{n}{a}p^a a!,
\]
gives
\[
(Q^{(m)})^{*k}
=\sum_{a=0}^{n}c_k^{(m)}(a)\,\widetilde Q_{a,p},
\quad \text{where } c_k^{(m)}(a):=\frac{\binom{n}{a}a!}{(n)_m^k}\,S_m(k,a).
\]
Using \eqref{eq:Sm-explicit},
\[
c_k^{(m)}(a)
=
\frac{1}{(n)_m^k}\binom{n}{a}
\sum_{j=0}^{a}(-1)^j\binom{a}{j}(a-j)_m^k.
\]

On the other hand, by inclusion--exclusion (equivalently \cite[Theorem~1, Eq.~(2.2)]{Stadje1990} with $S=[n]$ and $A=S$, so $s=l=n$),
\[
\Prob(A=a)
=
\binom{n}{a}\sum_{j=0}^{a}(-1)^j\binom{a}{j}
\left(\frac{\binom{a-j}{m}}{\binom{n}{m}}\right)^k.
\]
Since
\[
\binom{a-j}{m}=\frac{(a-j)_m}{m!},
\qquad
\binom{n}{m}=\frac{(n)_m}{m!},
\]
this becomes
\[
\Prob(A=a)
=
\frac{1}{(n)_m^k}\binom{n}{a}
\sum_{j=0}^{a}(-1)^j\binom{a}{j}(a-j)_m^k
=
c_k^{(m)}(a).
\]
By Definition~\ref{def:occupancy-m}, \(\mu_k^{(m)}(a)=\Prob(A=a)\), so
\[
\mu_k^{(m)}(a)
=
\frac{1}{(n)_m^k}\binom{n}{a}
\sum_{j=0}^{a}(-1)^j\binom{a}{j}(a-j)_m^k,
\]
which is \eqref{eq:mu-m-explicit}. Therefore
\[
(Q^{(m)})^{*k}
=\sum_{a=0}^{n}\mu_k^{(m)}(a)\,\widetilde Q_{a,p},
\]
which is \eqref{eq:Qm-mixture-a}. Since \(E=n-A\),
\[
P_k^{(m)}(u)=\Prob(E=u)=\Prob(A=n-u)=\mu_k^{(m)}(n-u),
\]
and \eqref{eq:Qm-mixture-u} follows.
\end{proof}

\begin{remark}
Definition~\ref{def:LR-general-mu} introduces \(\widetilde Q_{\mu,p}\) for an arbitrary pmf \(\mu\). Lemma~\ref{prop:Qm-mixture} identifies the shuffle law with this construction:
\[
(Q^{(m)})^{*k}=\widetilde Q_{\mu_k^{(m)},p},
\qquad
\mu_k^{(m)}(n-u)=P_k^{(m)}(u).
\]
Thus the lemma is the nontrivial step that connects top-\(m\)-to-random with the general mixture framework.
\end{remark}

Because \eqref{eq:Qm-mixture-u} uses the same nested sets \(A_u\) as before, the likelihood ratio again depends only on \(L_p\).

\begin{corollary}\label{cor:LR-top-m}
For every \(1\le m\le n\), \(k\ge0\), and \(x\in G_{n,p}\),
\begin{equation}\label{eq:LR-top-m}
\frac{(Q^{(m)})^{*k}(x)}{U(x)}
=\sum_{u=0}^{L_p(x)}P_k^{(m)}(u)\,u!\,p^u.
\end{equation}
\end{corollary}

\begin{proof}
By Lemma~\ref{prop:Qm-mixture},
\[
(Q^{(m)})^{*k}=\sum_{u=0}^{n}P_k^{(m)}(u)\,\widetilde Q_{n-u,p}.
\]
Now divide by \(U(x)\) and use \eqref{eq:component-ratio-indicator}.
\end{proof}

Equivalently, we may apply Remark~\ref{rem:LR-general-mu} with \(\mu(n-u)=P_k^{(m)}(u)\), which gives the same formula.

\begin{corollary}\label{cor:TV-exact-top-m}
Fix \(1\le m\le n\) and \(k\ge0\). If \((Q^{(m)})^{*k}=U\), then
\[
\|(Q^{(m)})^{*k}-U\|_{\TV}=0.
\]
Assume \((Q^{(m)})^{*k}\neq U\), and define
\[
w_{n,k}^{(m)}
:=
\min\left\{\ell\in\{0,\dots,n\}:\ \sum_{u=0}^{\ell}P_k^{(m)}(u)\,u!\,p^u>1\right\}.
\]
Then
\[
\|(Q^{(m)})^{*k}-U\|_{\TV}
=
1-\sum_{u=0}^{w_{n,k}^{(m)}-1}P_k^{(m)}(u)
+\frac{1}{w_{n,k}^{(m)}!\,p^{w_{n,k}^{(m)}}}
\left(\sum_{u=0}^{w_{n,k}^{(m)}-1}P_k^{(m)}(u)\,u!\,p^u-1\right).
\]
\end{corollary}

\begin{proof}
If \(k=0\), then
\[
(Q^{(m)})^{*0}=\delta_e=\widetilde Q_{0,p},
\]
so the result is Theorem~\ref{prop:TV-exact-general-mu} with \(\mu(0)=1\). Assume \(k\ge1\). By Lemma~\ref{prop:Qm-mixture},
\[
(Q^{(m)})^{*k}
=
\sum_{u=0}^{n}P_k^{(m)}(u)\,\widetilde Q_{n-u,p}
=
\widetilde Q_{\mu,p},
\qquad
\mu(n-u)=P_k^{(m)}(u).
\]
Now apply Theorem~\ref{prop:TV-exact-general-mu}.
\end{proof}

\subsection{Coupon collector}

For related recent work on coupon collecting with group drawings via Stein's method, see \cite{BetkenThaele2023}.

We now show that the number of never-chosen indices has the same Poisson limit as in the case \(m=1\), after rescaling time by \(1/m\).

\begin{lemma}\label{lem:empty-factorial-moments-m}
Fix \(n\ge1\), \(1\le m\le n\), and  \(k \ge 1\). In the \(m\)-sampling model,
\begin{equation}\label{eq:Em-factorial}
\E[(E)_u]
=(n)_u\left(\frac{\binom{n-u}{m}}{\binom{n}{m}}\right)^k
=(n)_u\left(\frac{(n-u)_m}{(n)_m}\right)^k
\qquad(0\le u\le n).
\end{equation}
\end{lemma}

\begin{proof}
Let
\[
I_i:=\mathbf 1_{\{i\text{ is never chosen}\}},
\qquad
E=\sum_{i=1}^n I_i.
\]
Since the falling factorial \((E)_u\) counts ordered distinct \(u\)-tuples of indices with \(I_i=1\), we have
\[
(E)_u
=\sum_{(i_1,\dots,i_u)\ \mathrm{distinct}} I_{i_1}\cdots I_{i_u}.
\]
Taking expectations,
\[
\E[(E)_u]
=\sum_{(i_1,\dots,i_u)\ \mathrm{distinct}}
\Prob(i_1,\dots,i_u\text{ are never chosen}).
\]
For fixed distinct \(i_1,\dots,i_u\), one round avoids them iff the chosen \(m\)-subset lies in their complement, so
\[
\Prob(\text{one round avoids }i_1,\dots,i_u)
=\frac{\binom{n-u}{m}}{\binom{n}{m}}.
\]
By independence across rounds,
\[
\Prob(i_1,\dots,i_u\text{ are never chosen})
=\left(\frac{\binom{n-u}{m}}{\binom{n}{m}}\right)^k.
\]
There are \((n)_u\) ordered \(u\)-tuples of distinct indices, giving \eqref{eq:Em-factorial}. The second form follows from \(\binom{x}{m}=(x)_m/m!\).
\end{proof}

\begin{proposition}[Poisson limit for \texorpdfstring{$m$}{m}-sampling]\label{lem:Pk-m-Poisson}
Fix \(m\ge1\), independent of \(n\), and \(c\in\R\). Write \(k=k_n(c)\).
Then, for each fixed \(u\ge0\),
\[
P_k^{(m)}(u)=\Prob(E=u)\longrightarrow e^{-\lambda}\frac{\lambda^u}{u!}.
\]
Equivalently,
\[
E\Rightarrow \mathrm{Poisson}(\lambda).
\]
Moreover, for each fixed \(u\ge0\),
\[
\E[(E)_u]\longrightarrow \lambda^u.
\]
\end{proposition}

\begin{proof}
In the notation of \cite{SchillingHenze}, \(Z_{n,m,1}(k)\) denotes the number of coupon types not observed after \(k\) independent drawings of uniform \(m\)-subsets of \([n]\). This is exactly our random variable \(E\). Therefore, by \cite[eq.~(4), Thm.~1]{SchillingHenze}, applied with their subset-size parameter \(s=m\), their count parameter \(c=1\), and their window parameter \(x=c\),
\[
E=Z_{n,m,1}(k)\Rightarrow \mathrm{Poisson}(\lambda).
\]
Also, by Lemma~\ref{lem:empty-factorial-moments-m},
\[
\E[(E)_u]=(n)_u\left(\frac{(n-u)_m}{(n)_m}\right)^k.
\]
For fixed \(u\),
\[
(n)_u=n^u(1+o(1)),
\qquad
\log\!\left(\frac{(n-u)_m}{(n)_m}\right)
=-\frac{um}{n}+O\!\left(\frac1{n^2}\right).
\]
Also, \(k=\frac{n}{m}(\log n+c)+O(1)\), so
\begin{align*}
k\log\!\left(\frac{(n-u)_m}{(n)_m}\right)
&=
\left(\frac{n}{m}(\log n+c)+O(1)\right)
\left(-\frac{um}{n}+O\!\left(\frac1{n^2}\right)\right)\\
&=
-u(\log n+c)+o(1).
\end{align*}
Therefore
\[
\left(\frac{(n-u)_m}{(n)_m}\right)^k
=
\exp\!\bigl(-u(\log n+c)+o(1)\bigr)
=
n^{-u}e^{-uc}(1+o(1))
=
n^{-u}\lambda^u(1+o(1)).
\]

Hence
\[
\E[(E)_u]\longrightarrow \lambda^u.
\]

Since both \(E\) and the Poisson limit are \(\mathbb Z_{\ge0}\)-valued, for each fixed \(u\ge0\),
\[
P_k^{(m)}(u)=\Prob(E=u)
=\Prob\!\left(E\le u+\tfrac12\right)-\Prob\!\left(E\le u-\tfrac12\right)
\longrightarrow e^{-\lambda}\frac{\lambda^u}{u!}.\qedhere
\]
\end{proof}

\begin{remark}
For \(m=1\), Proposition~\ref{lem:Pk-m-Poisson} is the classical empty-box Poisson limit used in \cite[Section~3]{DFP}.
\end{remark}

\subsection{TV profile}

For fixed \(m\), the limiting profile is the same \(f_p(c)\) as in the case \(m=1\); only the time scale changes from \(n\log n\) to \(\frac{n}{m}\log n\).

\begin{theorem}[Cutoff profile for colored top-\texorpdfstring{$m$}{m}-to-random]\label{thm:profile-top-m}
Fix \(c\in\R\) and integers \(p\ge1\) and \(m\ge1\), with \(m\) independent of \(n\). Let \(k=k_n(c)\).
Then
\[
\bigl\|(Q^{(m)})^{*k}-U\bigr\|_{\TV}
=f_p(c)+o(1)
\qquad(n\to\infty),
\]
where \(f_p(c)\) is the function from Theorem~\ref{prop:DFP3.1-colored}. In particular, the colored top-\(m\)-to-random shuffle has total-variation cutoff at time \(\frac{n}{m}\log n\) with window \(\frac{n}{m}\).
\end{theorem}

\begin{proof}
Let
\[
\mu_n:=\mu_k^{(m)},
\qquad\text{so}\qquad
\mu_n(n-u)=P_k^{(m)}(u).
\]
By Lemma~\ref{prop:Qm-mixture},
\[
(Q^{(m)})^{*k}
=\sum_{u=0}^{n}P_k^{(m)}(u)\,\widetilde Q_{n-u,p}
=\widetilde Q_{\mu_n,p}.
\]
By Proposition~\ref{lem:Pk-m-Poisson}, for each fixed \(u\ge0\),
\[
\mu_n(n-u)=P_k^{(m)}(u)\longrightarrow e^{-\lambda}\frac{\lambda^u}{u!}.
\]
Thus the hypotheses of Theorem~\ref{prop:DFP3.1-colored} hold, and
\[
\bigl\|(Q^{(m)})^{*k}-U\bigr\|_{\TV}\longrightarrow f_p(c).
\]
Since \(k_n(c)\) is nondecreasing in \(c\) and total variation distance to stationarity is nonincreasing in time, \(f_p\) is nonincreasing; together with the endpoint limits in Remark~\ref{rem:fp-endpoints}, this yields the cutoff statement.
\end{proof}

In particular, \(m=1\) recovers colored top-to-random, while \(m=p=1\) recovers \cite[Theorem~1.1]{DFP}; for \(p=1\) and fixed \(m\), the same profile holds with time scaled by \(1/m\).

\section{Separation, \texorpdfstring{$L^\infty(U)$}{Linfty(U)}, and likelihood-ratio functionals}\label{sec:sepdis}

This section evaluates separation, \(L^\infty(U)\), \(L^q(U)\), \(\chi^2\), and relative entropy using the same \(L_p\)-marginal reduction. It gives exact formulas, Poisson-regime profiles, monotonicity, and mixing-time consequences.

\subsection{Separation and \texorpdfstring{$L^\infty(U)$}{Linfty(U)}}

For probability measures \(M\) and \(U\) on a finite set with \(U(x)>0\) for all \(x\),
\[
\sep(M,U):=\max_x\Bigl(1-\frac{M(x)}{U(x)}\Bigr)
=1-\min_x\frac{M(x)}{U(x)}.
\]
Since
\[
U(x)=\frac{1}{p^n n!}>0
\qquad(x\in G_{n,p}),
\]
we may also write
\[
\left\|\frac{M(\cdot)}{U(\cdot)}-1\right\|_{L^\infty(U)}
=\max_{x\in G_{n,p}}\left|\frac{M(x)}{U(x)}-1\right|.
\]
Set
\[
\ell_{\min}:=\min\{L_p(x):x\in G_{n,p}\}.
\]

\begin{lemma}\label{lem:sep-general-mu}
With \(S_\mu\) as in \eqref{eq:def-Smu}, we have
\[
\sep(\widetilde Q_{\mu,p},U)=1-S_\mu(\ell_{\min}).
\]
\end{lemma}

\begin{proof}
By \eqref{eq:Smu},
\[
\frac{\widetilde Q_{\mu,p}(x)}{U(x)}=S_\mu(L_p(x)).
\]
Since \(S_\mu\) is nondecreasing (Lemma~\ref{lem:Smu-monotone}),
\[
\min_x\frac{\widetilde Q_{\mu,p}(x)}{U(x)}
=S_\mu\!\left(\min_x L_p(x)\right)
=S_\mu(\ell_{\min}).
\]
Therefore
\[
\sep(\widetilde Q_{\mu,p},U)
=1-\min_x\frac{\widetilde Q_{\mu,p}(x)}{U(x)}
=1-S_\mu(\ell_{\min}).\qedhere
\]
\end{proof}

\begin{lemma}\label{lem:ellmin}
We have
\[
\ell_{\min}=
\begin{cases}
1, & p=1,\\[4pt]
0, & p\ge2.
\end{cases}
\]
\end{lemma}

\begin{proof}
Since \(A_0=G_{n,p}\), we always have \(L_p(x)\ge0\).

If \(p=1\), then \(G_{n,1}=S_n\) and \(A_1=S_n\), since the order condition for a single label is vacuous. Hence \(L(\pi)\ge1\) for all \(\pi\in S_n\). If \(n=1\), then \(S_1=\{\mathrm{id}\}\) and \(L(\mathrm{id})=1\). If \(n\ge2\), the reverse permutation \(\mathrm{rev}=[n,n-1,\dots,1]\) satisfies \(\mathrm{rev}\in A_1\setminus A_2\), so \(L(\mathrm{rev})=1\). Thus \(\ell_{\min}=1\).

If \(p\ge2\), then
\[
A_1=\{(s,\sigma):s_{\sigma^{-1}(n)}=0\}.
\]
Choose \(x=(s,\mathrm{id})\) with \(s_n\neq0\). Then \(x\notin A_1\), so \(L_p(x)=0\). Hence \(\ell_{\min}=0\).
\end{proof}

\begin{corollary}\label{cor:sep-general-mu-explicit}
For
\[
\widetilde Q_{\mu,p}=\sum_{u=0}^{n}\mu(n-u)\,\widetilde Q_{n-u,p},
\]
we have
\[
\sep(\widetilde Q_{\mu,p},U)=
\begin{cases}
1-\mu(n)-\mu(n-1), & p=1,\\[4pt]
1-\mu(n), & p\ge2.
\end{cases}
\]
\end{corollary}

\begin{proof}
By Lemmas~\ref{lem:sep-general-mu} and~\ref{lem:ellmin},
\[
\sep(\widetilde Q_{\mu,p},U)=
\begin{cases}
1-S_\mu(1), & p=1,\\[4pt]
1-S_\mu(0), & p\ge2.
\end{cases}
\]
Now
\[
S_\mu(0)=\mu(n),
\qquad
S_\mu(1)=\mu(n)+\mu(n-1)
\quad(p=1),
\]
since \(1!\,p=1\) when \(p=1\).
\end{proof}

\begin{theorem}\label{thm:sep-poisson-general}
Assume \eqref{eq:mu-Poisson-assump}, and define
\[
s_p(c):=
\begin{cases}
1-e^{-\lambda}(1+\lambda), & p=1,\\[4pt]
1-e^{-\lambda}, & p\ge2.
\end{cases}
\]
Then
\[
\sep(\widetilde Q_{\mu_n,p},U)\longrightarrow s_p(c).
\]
\end{theorem}

\begin{proof}
By Corollary~\ref{cor:sep-general-mu-explicit},
\[
\sep(\widetilde Q_{\mu_n,p},U)=
\begin{cases}
1-\mu_n(n)-\mu_n(n-1), & p=1,\\[4pt]
1-\mu_n(n), & p\ge2.
\end{cases}
\]
The assumption gives
\[
\mu_n(n)\longrightarrow e^{-\lambda},
\qquad
\mu_n(n-1)\longrightarrow e^{-\lambda}\lambda.
\]
Substituting these limits yields the result.
\end{proof}

\begin{corollary}\label{cor:sep-profile-asymp}
Let \(s_p(c)\) be as in Theorem~\ref{thm:sep-poisson-general}. As \(c\to+\infty\),
\[
s_p(c)=
\begin{cases}
\displaystyle \frac12 e^{-2c}+O(e^{-3c}), & p=1,\\[8pt]
\displaystyle e^{-c}+O(e^{-2c}), & p\ge2.
\end{cases}
\]
Also, as \(c\to-\infty\) (equivalently \(\lambda\to\infty\)),
\[
1-s_p(c)\sim
\begin{cases}
\lambda e^{-\lambda}=e^{-c}e^{-e^{-c}}, & p=1,\\[4pt]
e^{-\lambda}=e^{-e^{-c}}, & p\ge2.
\end{cases}
\]
In particular,
\[
1-s_p(c)=\exp\bigl(-(1+o(1))e^{-c}\bigr)
\qquad(c\to-\infty)
\]
for every fixed \(p\ge1\).
\end{corollary}

\begin{proof}
If \(p=1\), then
\[
s_1(c)=1-e^{-\lambda}(1+\lambda).
\]
As \(\lambda\to0\),
\[
e^{-\lambda}(1+\lambda)=1-\frac{\lambda^2}{2}+O(\lambda^3),
\]
so
\[
s_1(c)=\frac{\lambda^2}{2}+O(\lambda^3)
=\frac12 e^{-2c}+O(e^{-3c}).
\]
If \(p\ge2\), then
\[
s_p(c)=1-e^{-\lambda}=\lambda+O(\lambda^2)=e^{-c}+O(e^{-2c})
\qquad(\lambda\to0).
\]

As \(\lambda\to\infty\), if \(p=1\), then
\[
1-s_1(c)=e^{-\lambda}(1+\lambda)\sim \lambda e^{-\lambda}.
\]
If \(p\ge2\), then
\[
1-s_p(c)=e^{-\lambda}.
\]
Also,
\[
\lambda e^{-\lambda}
=
\exp(-\lambda+\log \lambda)
=
\exp\bigl(-(1+o(1))\lambda\bigr)
\qquad(\lambda\to\infty),
\]
since \(\log \lambda=o(\lambda)\). Thus in both cases
\[
1-s_p(c)=\exp\bigl(-(1+o(1))\lambda\bigr).
\]
Replacing \(\lambda\) by \(e^{-c}\) gives the result.
\end{proof}

\begin{remark}
The pointwise Poisson assumption \eqref{eq:mu-Poisson-assump} controls the fixed low-order weights \(\mu_n(n-u)\). This is enough for separation, which depends only on \(\mu_n(n)\) and, when \(p=1\), also on \(\mu_n(n-1)\); see Corollary~\ref{cor:sep-general-mu-explicit}. It is also enough for total variation, since the proof of Theorem~\ref{prop:DFP3.1-colored} shows that the threshold \(w_{\mu_n}\) is eventually in \(\{w^*(c,p)-1,w^*(c,p)\}\), so only finitely many low-order weights matter.

By contrast, \(L^\infty(U)\), \(L^q(U)\) for \(q>1\), \(\chi^2\), and relative entropy are sensitive to tails. They involve the amplified quantities
\[
S_{\mu_n}(\ell)=\sum_{u=0}^{\ell}\mu_n(n-u)\,u!\,p^u,
\]
summed over all \(\ell\), or, for \(L^\infty(U)\), over all \(u\). Thus pointwise convergence in fixed \(u\) does not prevent mass from escaping to large \(u\), where the factor \(u!\,p^u\) can be large. This is why the profile results for these metrics require the additional domination hypotheses used in Theorems~\ref{thm:Linfty-poisson-general} and \ref{thm:Phi-top-m-asymp} in their finite-convergence regimes.
\end{remark}

\begin{theorem}[Exact and asymptotic separation for colored top-\texorpdfstring{$m$}{m}-to-random]\label{thm:sep-top-m}
Fix \(p\ge1\) and \(1\le m\le n\). Then, for every \(k\ge0\),
\[
\sep\bigl((Q^{(m)})^{*k},U\bigr)=
\begin{cases}
1-P_k^{(m)}(0)-P_k^{(m)}(1), & p=1,\\[4pt]
1-P_k^{(m)}(0), & p\ge2.
\end{cases}
\]
Moreover, if \(m\) is fixed, independent of \(n\), and \(c\in\R\), and \(k=k_n(c)\),
then
\[
\sep\bigl((Q^{(m)})^{*k},U\bigr)\longrightarrow
s_p(c)=\begin{cases}
1-e^{-\lambda}(1+\lambda), & p=1,\\[4pt]
1-e^{-\lambda}, & p\ge2.
\end{cases}
\]
\end{theorem}

\begin{proof}
If \(k=0\), then
\[
(Q^{(m)})^{*0}=\delta_e=\widetilde Q_{0,p},
\]
so the exact formula follows immediately from Corollary~\ref{cor:sep-general-mu-explicit}. Assume henceforth that \(k\ge1\).

Take \(\mu=\mu_k^{(m)}\), so
\[
(Q^{(m)})^{*k}
=\widetilde Q_{\mu,p},
\qquad
\mu(n-u)=P_k^{(m)}(u).
\]

The exact formula is Corollary~\ref{cor:sep-general-mu-explicit}. By Proposition~\ref{lem:Pk-m-Poisson},
\[
\mu(n-u)=P_k^{(m)}(u)\longrightarrow e^{-\lambda}\frac{\lambda^u}{u!}
\qquad(u\ge0\ \text{fixed}).
\]
Hence Theorem~\ref{thm:sep-poisson-general} applies.
\end{proof}

The following results connect the separation distance with the total variation distance.

\begin{theorem}[TV--separation comparison for nested-set mixtures]\label{thm:TV-vs-sep-general-mu}
For every nested-set mixture \(\widetilde Q_{\mu,p}\),
\[
\|\widetilde Q_{\mu,p}-U\|_{\TV}\le \sep(\widetilde Q_{\mu,p},U).
\]
Moreover,
\[
\frac12\,\sep(\widetilde Q_{\mu,1},U)\le \|\widetilde Q_{\mu,1}-U\|_{\TV},
\]
and, for \(p\ge2\),
\[
\left(1-\frac1p\right)\sep(\widetilde Q_{\mu,p},U)\le \|\widetilde Q_{\mu,p}-U\|_{\TV}.
\]
\end{theorem}

\begin{proof}
For the upper bound, apply Proposition~\ref{prop:distance-comparison} to the right random walk on \(G_{n,p}\) with one-step law \(\widetilde Q_{\mu,p}\), at time \(k=1\) and starting point \(e\). Since
\[
P^1(e,\cdot)=\widetilde Q_{\mu,p},
\]
it gives
\[
\|\widetilde Q_{\mu,p}-U\|_{\TV}\le \sep(\widetilde Q_{\mu,p},U).
\]

For the lower bounds, use
\[
\|\widetilde Q_{\mu,p}-U\|_{\TV}
=\max_{A\subseteq G_{n,p}}|\widetilde Q_{\mu,p}(A)-U(A)|.
\]

If \(p=1\) and \(n=1\), then \(G_{1,1}\) is a singleton, so both sides are \(0\). Assume \(p=1\) and \(n\ge2\), and take \(A=A_2\). Since \(A_0=A_1=S_n\),
\[
\widetilde Q_{n,1}=\widetilde Q_{n-1,1}=U.
\]
Also,
\[
A_2=\{\pi\in S_n:\pi^{-1}(n-1)<\pi^{-1}(n)\},
\]
so
\[
\widetilde Q_{n,1}(A_2)=\widetilde Q_{n-1,1}(A_2)=U(A_2)=\frac12.
\]
For \(u\ge2\), we have \(A_u\subseteq A_2\), hence
\[
\widetilde Q_{n-u,1}(A_2)=1.
\]
Therefore, by Definition~\ref{def:LR-general-mu},
\[
\widetilde Q_{\mu,1}(A_2)
=\sum_{u=0}^{n}\mu(n-u)\,\widetilde Q_{n-u,1}(A_2)
=\frac12\mu(n)+\frac12\mu(n-1)+\sum_{u=2}^{n}\mu(n-u)
=1-\frac12\bigl(\mu(n)+\mu(n-1)\bigr).
\]
Thus
\[
\widetilde Q_{\mu,1}(A_2)-U(A_2)
=\frac12\bigl(1-\mu(n)-\mu(n-1)\bigr)
=\frac12\,\sep(\widetilde Q_{\mu,1},U),
\]
by Corollary~\ref{cor:sep-general-mu-explicit}.

If \(p\ge2\), take \(A=A_1\). Then
\[
\widetilde Q_{n,p}(A_1)=\frac{|A_1|}{|A_0|}=U(A_1)=\frac1p,
\qquad
\widetilde Q_{n-u,p}(A_1)=1
\quad(u\ge1).
\]
Thus
\[
\widetilde Q_{\mu,p}(A_1)
=\frac1p\,\mu(n)+\sum_{u=1}^{n}\mu(n-u)
=1-\left(1-\frac1p\right)\mu(n),
\]
while
\[
U(A_1)=\frac1p.
\]
Hence
\[
\widetilde Q_{\mu,p}(A_1)-U(A_1)
=\left(1-\frac1p\right)(1-\mu(n))
=\left(1-\frac1p\right)\sep(\widetilde Q_{\mu,p},U),
\]
again by Corollary~\ref{cor:sep-general-mu-explicit}.
\end{proof}

\begin{corollary}[Exact TV--separation identities in the late window]\label{cor:TV-sep-ratio}
Let \(\widetilde Q_{\mu,p}\) be a nested-set mixture. If \(p\ge2\) and \(S_\mu(1)>1\), then
\[
\|\widetilde Q_{\mu,p}-U\|_{\TV}
=
\left(1-\frac1p\right)\sep(\widetilde Q_{\mu,p},U).
\]
If \(p=1\), \(n\ge2\), and \(S_\mu(2)>1\), then
\[
\|\widetilde Q_{\mu,1}-U\|_{\TV}
=
\frac12\,\sep(\widetilde Q_{\mu,1},U).
\]

Consequently, under \eqref{eq:mu-Poisson-assump} with \(c\ge0\), for all sufficiently large \(n\),
\[
\|M_n-U\|_{\TV}
=
\begin{cases}
\displaystyle \frac12\,\sep(M_n,U), & p=1,\\[8pt]
\displaystyle \left(1-\frac1p\right)\sep(M_n,U), & p\ge2.
\end{cases}
\]
In particular, for fixed \(m\ge1\), fixed \(c\ge0\), and \(k=k_n^{(m)}(c)\), the same identities hold with \(M_n=(Q^{(m)})^{*k}\) for all sufficiently large \(n\).
\end{corollary}

\begin{proof}
If \(p\ge2\), then \(S_\mu(0)=\mu(n)\le1\), so \(S_\mu(1)>1\) implies \(w_\mu=1\). By Theorem~\ref{prop:TV-exact-general-mu},
\[
\|\widetilde Q_{\mu,p}-U\|_{\TV}
=
1-\mu(n)+\frac1p\bigl(\mu(n)-1\bigr)
=
\left(1-\frac1p\right)(1-\mu(n)).
\]
Corollary~\ref{cor:sep-general-mu-explicit} now proves the first identity.

If \(p=1\), then \(S_\mu(1)=\mu(n)+\mu(n-1)\le1\), so \(S_\mu(2)>1\) implies \(w_\mu=2\). By Theorem~\ref{prop:TV-exact-general-mu},
\[
\begin{aligned}
\|\widetilde Q_{\mu,1}-U\|_{\TV}
&=
1-\mu(n)-\mu(n-1)
+\frac12\bigl(\mu(n)+\mu(n-1)-1\bigr)\\
&=
\frac12\bigl(1-\mu(n)-\mu(n-1)\bigr).
\end{aligned}
\]
Corollary~\ref{cor:sep-general-mu-explicit} now proves the second identity.

For the Poisson-regime statement, by \eqref{eq:Smu-pointwise-limit},
\[
S_{\mu_n}(1)\to e^{-\lambda}(1+p\lambda)
\qquad(p\ge2),
\]
and
\[
S_{\mu_n}(2)\to e^{-\lambda}(1+\lambda+\lambda^2)
\qquad(p=1).
\]
If \(c\ge0\), then \(0<\lambda\le1\). For \(p\ge2\),
\[
e^\lambda<1+2\lambda\le1+p\lambda,
\]
so \(e^{-\lambda}(1+p\lambda)>1\). For \(p=1\),
\[
e^\lambda<1+\lambda+\lambda^2,
\]
so \(e^{-\lambda}(1+\lambda+\lambda^2)>1\). Thus the relevant hypothesis holds for all sufficiently large \(n\). The top-\(m\) statement follows by taking \(M_n=(Q^{(m)})^{*k}\) and using Lemma~\ref{prop:Qm-mixture} and Proposition~\ref{lem:Pk-m-Poisson}.
\end{proof}

\begin{corollary}[Separation mixing time]\label{cor:sep-mixing-time}
Fix \(p\ge1\) and \(m\ge1\), independent of \(n\). Let
\[
t_{\sep}^{(m)}(\varepsilon)
:=\min\left\{k\ge0:\sep\bigl((Q^{(m)})^{*k},U\bigr)\le\varepsilon\right\},
\qquad 0<\varepsilon<1.
\]
Let \(c_{\varepsilon,p}\) be the unique real number satisfying
\[
s_p(c_{\varepsilon,p})=\varepsilon.
\]
Then
\[
t_{\sep}^{(m)}(\varepsilon)
=
\frac{n}{m}\bigl(\log n+c_{\varepsilon,p}\bigr)+o(n).
\]
For \(p\ge2\),
\[
c_{\varepsilon,p}=-\log\bigl(-\log(1-\varepsilon)\bigr).
\]
\end{corollary}

\begin{proof}
The function \(s_p\) is continuous and strictly decreasing from \(1\) to \(0\), so \(c_{\varepsilon,p}\) exists and is unique. Fix \(\delta>0\), and set
\[
k_n^\pm:=k_n(c_{\varepsilon,p}\pm\delta).
\]
By Theorem~\ref{thm:sep-top-m},
\[
\sep\bigl((Q^{(m)})^{*k_n^-},U\bigr)\to s_p(c_{\varepsilon,p}-\delta)>\varepsilon,
\]
and
\[
\sep\bigl((Q^{(m)})^{*k_n^+},U\bigr)\to s_p(c_{\varepsilon,p}+\delta)<\varepsilon.
\]
By the standard monotonicity of separation distance \cite[Exercise~6.4]{LPW}, for all sufficiently large \(n\),
\[
k_n^-<t_{\sep}^{(m)}(\varepsilon)\le k_n^+.
\]
Since \(k_n(c)=\frac{n}{m}(\log n+c)+O(1)\), these bounds imply
\[
-\frac{\delta}{m}+o(1)
\le
\frac{1}{n}\left(t_{\sep}^{(m)}(\varepsilon)
-\frac{n}{m}(\log n+c_{\varepsilon,p})\right)
\le
\frac{\delta}{m}+o(1).
\]
First letting \(n\to\infty\) with \(\delta\) fixed, and then letting \(\delta\downarrow0\), gives
\[
t_{\sep}^{(m)}(\varepsilon)
=
\frac{n}{m}\bigl(\log n+c_{\varepsilon,p}\bigr)+o(n).
\]
For \(p\ge2\), \(s_p(c)=1-e^{-e^{-c}}\), which gives the displayed formula.
\end{proof}

\begin{remark}
A corresponding asymptotic for the total-variation mixing time also follows from Theorem~\ref{thm:profile-top-m} and the later monotonicity of \(f_p\); unlike separation, however, it is not given by a comparably simple closed formula, since \(f_p\) is piecewise through \(w^*(c,p)\).
\end{remark}

We now turn from separation to \(\,L^\infty(U)\,\) and related likelihood-ratio functionals. Recall that
\[
\mu_k^{(m)}(n-u)=P_k^{(m)}(u),
\qquad
S_{\mu_k^{(m)}}(\ell)=\sum_{u=0}^{\ell}P_k^{(m)}(u)\,u!\,p^u.
\]

\begin{theorem}[Exact \texorpdfstring{$L^\infty(U)$}{Linfty(U)} formula]\label{thm:Linfty-general-mu}
We have
\[
\left\|\frac{\widetilde Q_{\mu,p}(\cdot)}{U(\cdot)}-1\right\|_{L^\infty(U)}
=S_\mu(n)-1
=\sum_{u=0}^{n}\mu(n-u)\,u!\,p^u-1.
\]
\end{theorem}

\begin{proof}
Since \(e\in A_n\), we have \(L_p(e)=n\), so
\[
\max_x\frac{\widetilde Q_{\mu,p}(x)}{U(x)}=S_\mu(n)
\]
by \eqref{eq:Smu} and the monotonicity of \(S_\mu\). By Lemma~\ref{lem:ellmin},
\[
\min_x\frac{\widetilde Q_{\mu,p}(x)}{U(x)}=S_\mu(\ell_{\min}).
\]
Hence
\[
\left\|\frac{\widetilde Q_{\mu,p}(\cdot)}{U(\cdot)}-1\right\|_{L^\infty(U)}
=\max\{S_\mu(n)-1,\ 1-S_\mu(\ell_{\min})\}.
\]

If \(p\ge2\), then \(\ell_{\min}=0\) (Lemma \ref{lem:ellmin}), so
\[
S_\mu(n)-1
=\sum_{u=1}^{n}\mu(n-u)\,(u!\,p^u-1)
\ge \sum_{u=1}^{n}\mu(n-u)
=1-\mu(n)
=1-S_\mu(0).
\]
If \(p=1\), then \(\ell_{\min}=1\) (Lemma \ref{lem:ellmin}), so
\[
S_\mu(n)-1
=\sum_{u=2}^{n}\mu(n-u)\,(u!-1)
\ge \sum_{u=2}^{n}\mu(n-u)
=1-\mu(n)-\mu(n-1)
=1-S_\mu(1).
\]
Thus, in both cases, \(S_\mu(n)-1\ge 1-S_\mu(\ell_{\min})\), so
\[
\left\|\frac{\widetilde Q_{\mu,p}(\cdot)}{U(\cdot)}-1\right\|_{L^\infty(U)}
=S_\mu(n)-1.\qedhere
\]
\end{proof}

\begin{corollary}[Exact \texorpdfstring{$L^\infty(U)$}{Linfty(U)} formula for colored top-\texorpdfstring{$m$}{m}-to-random]\label{cor:Linfty-top-m-exact}
For every \(k\ge0\),
\begin{equation}\label{eq:Linfty-top-m-exact}
\left\|\frac{(Q^{(m)})^{*k}(\cdot)}{U(\cdot)}-1\right\|_{L^\infty(U)}
=S_{\mu_k^{(m)}}(n)-1
=\sum_{u=0}^{n}P_k^{(m)}(u)\,u!\,p^u-1.
\end{equation}
If \(k\ge1\), then the maximum of
\[
x\longmapsto \frac{(Q^{(m)})^{*k}(x)}{U(x)}
\]
is attained exactly on
\[
A_{n-m}=\supp(B_m).
\]
\end{corollary}

\begin{proof}
The identity \eqref{eq:Linfty-top-m-exact} is Theorem~\ref{thm:Linfty-general-mu} applied to
\[
(Q^{(m)})^{*k}=\widetilde Q_{\mu_k^{(m)},p}.
\]

Assume \(k\ge1\). Since at least \(m\) indices are chosen after one round,
\[
P_k^{(m)}(u)=0\qquad(u>n-m).
\]
Fix an \(m\)-subset \(T\subseteq[n]\). The event
\[
S_1=\cdots=S_k=T
\]
has positive probability, and on this event \(E=n-m\). Hence
\[
P_k^{(m)}(n-m)>0.
\]
Thus
\[
r_{\mu_k^{(m)}}:=\max\{u:\mu_k^{(m)}(n-u)>0\}=n-m.
\]
Now apply Lemma~\ref{lem:Smu-plateau-max}.
\end{proof}

\begin{lemma}\label{lem:Linfty-domination}
Assume $k\ge1$. For \(0\le u\le n\),
\[
P_k^{(m)}(u)\,u!
\le \E[(E)_u]
=(n)_u\left(\frac{(n-u)_m}{(n)_m}\right)^k
\le (n e^{-mk/n})^u.
\]
Also,
\[
P_k^{(m)}(u)=0
\qquad(u>n-m).
\]
\end{lemma}

\begin{proof}
On \(\{E=u\}\), we have \((E)_u=u!\), and always \((E)_u\ge0\). Thus
\[
u!\,P_k^{(m)}(u)\le \E[(E)_u].
\]
The exact formula for \(\E[(E)_u]\) is Lemma~\ref{lem:empty-factorial-moments-m}.

Also,
\[
(n)_u\le n^u.
\]
If \(u>n-m\), then \((n-u)_m=0\), so
\[
\left(\frac{(n-u)_m}{(n)_m}\right)^k=0\le e^{-umk/n}.
\]
If \(0\le u\le n-m\), then for \(0\le j\le m-1\),
\[
0\le 1-\frac{u}{n-j}\le e^{-u/(n-j)}.
\]
Hence
\[
\frac{(n-u)_m}{(n)_m}
=\prod_{j=0}^{m-1}\left(1-\frac{u}{n-j}\right)
\le \exp\!\left(-u\sum_{j=0}^{m-1}\frac{1}{n-j}\right)
\le e^{-um/n},
\]
since \(\sum_{j=0}^{m-1}(n-j)^{-1}\ge m/n\). Therefore
\[
\left(\frac{(n-u)_m}{(n)_m}\right)^k\le e^{-umk/n}
\]
for all \(0\le u\le n\). Combining the bounds gives
\[
\E[(E)_u]\le n^u e^{-umk/n}=(n e^{-mk/n})^u.
\]
Finally, since \(k\ge1\), after one round at least \(m\) indices have been chosen, so \(E\le n-m\) almost surely. Hence 
\[ P_k^{(m)}(u)=0 \qquad(u>n-m).\qedhere \]
\end{proof}

\begin{theorem}[Poisson-regime \texorpdfstring{$L^\infty(U)$}{Linfty(U)} profile]\label{thm:Linfty-poisson-general}
Assume \eqref{eq:mu-Poisson-assump}, and define

\[
H_{\infty,p}(c):=
\begin{cases}
\displaystyle \frac{e^{-\lambda}}{1-r}-1, & r<1 \ \ (\text{equivalently }c>\log p),\\[10pt]
+\infty, & r\ge1 \ \ (\text{equivalently }c\le \log p).
\end{cases}
\]
Assume further, in the case \(r<1\), that there exists a summable \((b_u)_{u\ge0}\) such that, for all sufficiently large \(n\) and all \(u\ge0\),
\[
0\le \mu_n(n-u)\,u!\,p^u\le b_u.
\]
Then, with convergence in the extended interval \([0,\infty]\),
\[
\left\|\frac{M_n(\cdot)}{U(\cdot)}-1\right\|_{L^\infty(U)}
\longrightarrow
H_{\infty,p}(c).
\]
\end{theorem}

\begin{proof}
By Theorem~\ref{thm:Linfty-general-mu} and our standing convention
\(\mu_n(a)=0\) for \(a\notin\{0,\dots,n\}\),
\[
\left\|\frac{M_n(\cdot)}{U(\cdot)}-1\right\|_{L^\infty(U)}
=
\sum_{u=0}^{\infty}\mu_n(n-u)\,u!\,p^u-1.
\]

If \(r<1\), dominated convergence gives
\[
\sum_{u=0}^{\infty}\mu_n(n-u)\,u!\,p^u
\longrightarrow
\sum_{u=0}^{\infty}e^{-\lambda}(p\lambda)^u
=
\frac{e^{-\lambda}}{1-r},
\]
and hence
\[
\left\|\frac{M_n(\cdot)}{U(\cdot)}-1\right\|_{L^\infty(U)}
\longrightarrow
\frac{e^{-\lambda}}{1-r}-1.
\]

If \(r\ge1\), then for every fixed \(M\),
\[
\sum_{u=0}^{\infty}\mu_n(n-u)\,u!\,p^u
\ge
\sum_{u=0}^{M}\mu_n(n-u)\,u!\,p^u.
\]
Taking \(\liminf\) and using the assumed pointwise convergence,
\[
\liminf_{n\to\infty}\sum_{u=0}^{\infty}\mu_n(n-u)\,u!\,p^u
\ge
e^{-\lambda}\sum_{u=0}^{M}r^u.
\]
Let \(M\to\infty\). Since \(r\ge1\), the right-hand side diverges. Therefore
\[
\left\|\frac{M_n(\cdot)}{U(\cdot)}-1\right\|_{L^\infty(U)}
\longrightarrow +\infty.\qedhere
\]
\end{proof}

\begin{corollary}\label{cor:Linfty-top-m-limit}
Fix \(p\ge1\), \(m\ge1\), and \(c\in\R\), with \(m\) independent of \(n\). Let \(k=k_n(c)\).
Then, with convergence in the extended interval \([0,\infty]\),
\[
\left\|\frac{(Q^{(m)})^{*k}(\cdot)}{U(\cdot)}-1\right\|_{L^\infty(U)}
\longrightarrow H_{\infty,p}(c).
\]
Equivalently, after re-centering the window parameter, write
\[
k=k_n(d+\log p)=\Bigl\lfloor \frac{n}{m}(\log n+\log p+d)\Bigr\rfloor,
\]
so that \(r=e^{-d}\); thus the transition is shifted by \(\frac{n}{m}\log p\).
\end{corollary}

\begin{proof}
Set \(\mu_n(n-u):=P_k^{(m)}(u)\). By Lemma~\ref{prop:Qm-mixture},
\[
M_n:=(Q^{(m)})^{*k}=\widetilde Q_{\mu_n,p}.
\]
By Proposition~\ref{lem:Pk-m-Poisson}, \(\mu_n(n-u)\to e^{-\lambda}\lambda^u/u!\) for each fixed \(u\ge0\), so the pointwise Poisson assumption of Theorem~\ref{thm:Linfty-poisson-general} holds.

If \(r<1\), choose \(R\in(r,1)\). By Lemma~\ref{lem:Linfty-domination},
\[
0\le P_k^{(m)}(u)\,u!\,p^u
\le \bigl(p\,n e^{-mk/n}\bigr)^u.
\]
Since
\[
p\,n e^{-mk/n}\longrightarrow p e^{-c}=r,
\]
we have \(p\,n e^{-mk/n}\le R\) for all sufficiently large \(n\). Hence
\[
0\le P_k^{(m)}(u)\,u!\,p^u\le R^u
\qquad(u\ge0),
\]
and Theorem~\ref{thm:Linfty-poisson-general} applies with
\[
b_u:=R^u.
\]

If \(r\ge1\), the same theorem applies without any further hypothesis.
\end{proof}

\begin{corollary}[Uniform mixing time]\label{cor:uniform-mixing-time}
Fix \(p\ge1\), \(m\ge1\), independent of \(n\), and \(0<\varepsilon<\infty\). Let
\[
t_{\infty}^{(m)}(\varepsilon)
:=
\min\left\{k\ge0:
\left\|\frac{(Q^{(m)})^{*k}(\cdot)}{U(\cdot)}-1\right\|_{L^\infty(U)}
\le \varepsilon
\right\}.
\]
Let \(d_{\varepsilon,p}>0\) be the unique solution of
\[
\frac{\exp(-e^{-d_{\varepsilon,p}}/p)}{1-e^{-d_{\varepsilon,p}}}-1
=
\varepsilon.
\]
Then the uniform mixing time satisfies
\[
t_{\infty}^{(m)}(\varepsilon)
=
\frac{n}{m}\left(\log n+\log p+d_{\varepsilon,p}\right)+o(n).
\]
\end{corollary}

\begin{proof}
For \(d\in\mathbb R\), set
\[
k_n^\infty(d):=
\left\lfloor\frac{n}{m}(\log n+\log p+d)\right\rfloor
=
k_n(\log p+d).
\]
Applying Corollary~\ref{cor:Linfty-top-m-limit} with \(c=\log p+d\), we have
\[
\lambda=e^{-c}=\frac{e^{-d}}{p},
\qquad
r=p\lambda=e^{-d}.
\]
Hence
\[
\left\|\frac{(Q^{(m)})^{*k_n^\infty(d)}(\cdot)}{U(\cdot)}-1\right\|_{L^\infty(U)}
\longrightarrow
\begin{cases}
\displaystyle \frac{\exp(-e^{-d}/p)}{1-e^{-d}}-1, & d>0,\\[8pt]
+\infty, & d\le0.
\end{cases}
\]
The finite profile is continuous and strictly decreasing from \(+\infty\) to \(0\) on \(d>0\), so \(d_{\varepsilon,p}\) is uniquely defined. By the monotonicity of the \(L^\infty(U)\)-distance in time, proved below in Proposition~\ref{prop:Lq-KL-monotone}, the same squeeze argument as in Corollary~\ref{cor:sep-mixing-time} gives the result.
\end{proof}

\subsection{Integrated functionals of the likelihood ratio}

For \(\Phi:[0,\infty)\to\R\),
\[
\sum_{x\in G_{n,p}}U(x)\,\Phi\!\left(\frac{\widetilde Q_{\mu,p}(x)}{U(x)}\right)
\]
depends only on the \(L_p\)-marginal.

\begin{theorem}[Reduction to the \texorpdfstring{$L_p$}{Lp}-marginal]\label{thm:Phi-general-mu}
For every \(\Phi:[0,\infty)\to\R\),
\[
\sum_{x\in G_{n,p}}U(x)\,\Phi\!\left(\frac{\widetilde Q_{\mu,p}(x)}{U(x)}\right)
=
\sum_{\ell=0}^{n}U\{L_p=\ell\}\,\Phi\!\bigl(S_\mu(\ell)\bigr).
\]
Equivalently,
\[
\sum_{x\in G_{n,p}}U(x)\,\Phi\!\left(\frac{\widetilde Q_{\mu,p}(x)}{U(x)}\right)
=
\sum_{\ell=0}^{n-1}
\left(\frac{1}{\ell!\,p^\ell}-\frac{1}{(\ell+1)!\,p^{\ell+1}}\right)\Phi\!\bigl(S_\mu(\ell)\bigr)
+\frac{1}{n!\,p^n}\,\Phi\!\bigl(S_\mu(n)\bigr).
\]
\end{theorem}

\begin{proof}
By \eqref{eq:Smu},
\[
\frac{\widetilde Q_{\mu,p}(x)}{U(x)}=S_\mu(L_p(x)).
\]
Partition \(G_{n,p}\) by the level sets of \(L_p\), then use Lemma~\ref{lem:Lp-under-U}.
\end{proof}

\begin{remark}
By Lemma~\ref{lem:Smu-plateau-max}, one may combine all terms with \(\ell\ge r_\mu\) into a single tail term. For the \(k\ge1\) top-\(m\) law, \(r_\mu=n-m\). This is useful in small examples, but does not materially simplify the general formula.
\end{remark}

In particular, total variation is recovered by the choice $\Phi(t)=(t-1)_+$ by Proposition~\ref{prop:TV-likelihood} with $M=\widetilde Q_{\mu,p}$.

\begin{corollary}\label{cor:Phi-top-m-exact}
We have
\[
\sum_{x\in G_{n,p}}U(x)\,\Phi\!\left(\frac{(Q^{(m)})^{*k}(x)}{U(x)}\right)
=
\sum_{\ell=0}^{n}U\{L_p=\ell\}\,\Phi\!\bigl(S_{\mu_k^{(m)}}(\ell)\bigr).
\]
Equivalently,
\[
\sum_{x\in G_{n,p}}U(x)\,\Phi\!\left(\frac{(Q^{(m)})^{*k}(x)}{U(x)}\right)
=
\sum_{\ell=0}^{n-1}
\left(\frac{1}{\ell!\,p^\ell}-\frac{1}{(\ell+1)!\,p^{\ell+1}}\right)\Phi\!\bigl(S_{\mu_k^{(m)}}(\ell)\bigr)
+\frac{1}{n!\,p^n}\,\Phi\!\bigl(S_{\mu_k^{(m)}}(n)\bigr).
\]
\end{corollary}

\begin{proof}
Apply Theorem~\ref{thm:Phi-general-mu} to
\[
(Q^{(m)})^{*k}=\widetilde Q_{\mu_k^{(m)},p}.\qedhere
\]
\end{proof}

Since \(U(x)>0\) for all \(x\in G_{n,p}\), for \(1\le q<\infty\),
\[
\left\|\frac{M(\cdot)}{U(\cdot)}-1\right\|_{L^q(U)}^q
=
\sum_{x\in G_{n,p}}U(x)\left|\frac{M(x)}{U(x)}-1\right|^q.
\]
For \(q=2\), define
\[
\chi^2(M,U):=
\left\|\frac{M(\cdot)}{U(\cdot)}-1\right\|_{L^2(U)}^2.
\]
Expanding the square,
\begin{equation}\label{eq:L2-via-second-moment}
\chi^2(M,U)
=
\sum_{x\in G_{n,p}}\frac{M(x)^2}{U(x)}-1
=
\sum_{x\in G_{n,p}}U(x)\left(\frac{M(x)}{U(x)}\right)^2-1.
\end{equation}

\begin{corollary}\label{cor:Lq-general-mu}
For \(1\le q<\infty\),
\[
\left\|\frac{\widetilde Q_{\mu,p}(\cdot)}{U(\cdot)}-1\right\|_{L^q(U)}^q
=
\sum_{\ell=0}^{n}U\{L_p=\ell\}\,|S_\mu(\ell)-1|^q.
\]
Equivalently,
\[
\left\|\frac{\widetilde Q_{\mu,p}(\cdot)}{U(\cdot)}-1\right\|_{L^q(U)}^q
=
\sum_{\ell=0}^{n-1}
\left(\frac{1}{\ell!\,p^\ell}-\frac{1}{(\ell+1)!\,p^{\ell+1}}\right)|S_\mu(\ell)-1|^q
+\frac{1}{n!\,p^n}|S_\mu(n)-1|^q.
\]
\end{corollary}

\begin{proof}
Apply Theorem~\ref{thm:Phi-general-mu} with \(\Phi(t)=|t-1|^q\).
\end{proof}

\begin{corollary}\label{cor:chisq-general-mu}
We have
\[
\chi^2(\widetilde Q_{\mu,p},U)
=
\sum_{u,v=0}^{n}\mu(n-u)\mu(n-v)\,\min(u,v)!\,p^{\min(u,v)}-1.
\]
\end{corollary}

\begin{proof}
By \eqref{eq:L2-via-second-moment} and Theorem~\ref{thm:Phi-general-mu} with \(\Phi(t)=t^2\),
\[
\chi^2(\widetilde Q_{\mu,p},U)+1
=
\sum_{x\in G_{n,p}}U(x)\left(\frac{\widetilde Q_{\mu,p}(x)}{U(x)}\right)^2
=
\sum_{\ell=0}^{n}U\{L_p=\ell\}\,S_\mu(\ell)^2.
\]
By \eqref{eq:def-Smu} and interchanging the finite sums,
\[
\chi^2(\widetilde Q_{\mu,p},U)+1
=
\sum_{u,v=0}^{n}\mu(n-u)\mu(n-v)\,u!\,p^u\,v!\,p^v
\sum_{\ell=\max(u,v)}^{n}U\{L_p=\ell\}.
\]
By Lemma~\ref{lem:Lp-under-U},
\[
\sum_{\ell=\max(u,v)}^{n}U\{L_p=\ell\}
=
U\{L_p\ge\max(u,v)\}
=
\frac{1}{\max(u,v)!\,p^{\max(u,v)}}.
\]
Hence the factor
\[
\frac{u!\,p^u\,v!\,p^v}{\max(u,v)!\,p^{\max(u,v)}}
=
\min(u,v)!\,p^{\min(u,v)}.
\]
Subtract \(1\).
\end{proof}

\begin{corollary}\label{cor:L2-top-m-exact}
We have
\[
\chi^2\bigl((Q^{(m)})^{*k},U\bigr)
=
\sum_{u,v=0}^{n}
P_k^{(m)}(u)\,P_k^{(m)}(v)\,\min(u,v)!\,p^{\min(u,v)}-1.
\]
\end{corollary}

\begin{proof}
Apply Corollary~\ref{cor:chisq-general-mu} to
\[
(Q^{(m)})^{*k}=\widetilde Q_{\mu_k^{(m)},p}. \qedhere
\]
\end{proof}

For probability measures \(M\) and \(U\) on \(G_{n,p}\), define
\[
D(M\|U):=\sum_{x\in G_{n,p}}M(x)\log\frac{M(x)}{U(x)},
\]
with the convention \(0\log0:=0\).

\begin{corollary}\label{cor:KL-general-mu}
We have
\[
D(\widetilde Q_{\mu,p}\|U)
=
\sum_{\ell=0}^{n}U\{L_p=\ell\}\,S_\mu(\ell)\log S_\mu(\ell).
\]
Equivalently,
\[
D(\widetilde Q_{\mu,p}\|U)
=
\sum_{\ell=0}^{n-1}
\left(\frac{1}{\ell!\,p^\ell}-\frac{1}{(\ell+1)!\,p^{\ell+1}}\right)
S_\mu(\ell)\log S_\mu(\ell)
+\frac{1}{n!\,p^n}S_\mu(n)\log S_\mu(n).
\]
\end{corollary}

\begin{proof}
Apply Theorem~\ref{thm:Phi-general-mu} with \(\Phi(t)=t\log t\) and \(\Phi(0)=0\).
\end{proof}

For \(p=1\), Corollary~\ref{cor:KL-general-mu} agrees with the reduction used by Stark \cite[Lemma~2.1]{Stark2002InformationLoss}.

\begin{corollary}\label{cor:KL-top-m-exact}
We have
\[
D\bigl((Q^{(m)})^{*k}\|U\bigr)
=
\sum_{\ell=0}^{n}U\{L_p=\ell\}\,
S_{\mu_k^{(m)}}(\ell)\log S_{\mu_k^{(m)}}(\ell).
\]
Equivalently,
\[
D\bigl((Q^{(m)})^{*k}\|U\bigr)
=
\sum_{\ell=0}^{n-1}
\left(\frac{1}{\ell!\,p^\ell}-\frac{1}{(\ell+1)!\,p^{\ell+1}}\right)
S_{\mu_k^{(m)}}(\ell)\log S_{\mu_k^{(m)}}(\ell)
+\frac{1}{n!\,p^n}S_{\mu_k^{(m)}}(n)\log S_{\mu_k^{(m)}}(n).
\]
\end{corollary}

\begin{proof}
Apply Corollary~\ref{cor:KL-general-mu} to
\[
(Q^{(m)})^{*k}=\widetilde Q_{\mu_k^{(m)},p}.\qedhere
\]
\end{proof}

\subsection{Poisson-regime profiles}

\begin{corollary}\label{cor:Lp-limit-top-m}
Assume \eqref{eq:mu-Poisson-assump}. Then, for fixed \(\ell\ge0\),
\[
M_n\{L_p=\ell\}
\longrightarrow
\left(\frac{1}{\ell!\,p^\ell}-\frac{1}{(\ell+1)!\,p^{\ell+1}}\right)
e^{-\lambda}\sum_{u=0}^{\ell}(p\lambda)^u.
\]
\end{corollary}

\begin{proof}
For all sufficiently large \(n\), we have \(\ell\le n-1\). By Corollary~\ref{cor:Lp-law-general-mu},
\[
M_n\{L_p=\ell\}
=
\left(\frac{1}{\ell!\,p^\ell}-\frac{1}{(\ell+1)!\,p^{\ell+1}}\right)
\sum_{u=0}^{\ell}\mu_n(n-u)\,u!\,p^u.
\]
Now pass to the limit termwise in the finite sum.
\end{proof}

\begin{theorem}[Poisson-regime limit for likelihood-ratio functionals]\label{thm:Phi-top-m-asymp}
Assume \eqref{eq:mu-Poisson-assump}. Let \(\Phi:[0,\infty)\to\R\) be continuous and satisfy
\begin{equation}\label{eq:Phi-growth}
|\Phi(t)|\le C_\Phi(1+t^d)
\qquad(t\ge0)
\end{equation}
for some \(C_\Phi>0\) and \(d\ge0\). Assume in addition that there exists \(B\ge1\) such that, for all sufficiently large \(n\) and all \(u\ge0\),
\[
0\le \mu_n(n-u)\,u!\,p^u\le B^u.
\]
Then
\[
\sum_{x\in G_{n,p}}U(x)\,\Phi\!\left(\frac{M_n(x)}{U(x)}\right)
\longrightarrow
\sum_{\ell=0}^{\infty}
\left(\frac{1}{\ell!\,p^\ell}-\frac{1}{(\ell+1)!\,p^{\ell+1}}\right)\Phi\!\bigl(s(\ell)\bigr).
\]
\end{theorem}

\begin{proof}
By Theorem~\ref{thm:Phi-general-mu},
\[
\sum_{x\in G_{n,p}}U(x)\,\Phi\!\left(\frac{M_n(x)}{U(x)}\right)
=
\sum_{\ell=0}^{\infty} b_{n,\ell},
\]
where, extending the finite sum by zeros,
\[
b_{n,\ell}:=
\begin{cases}
U\{L_p=\ell\}\,\Phi\!\bigl(S_{\mu_n}(\ell)\bigr), & 0\le \ell\le n,\\[4pt]
0, & \ell>n,
\end{cases}
\]
and \(S_{\mu_n}\) is given by \eqref{eq:def-Smu}.

Fix \(\ell\ge0\). By the assumed pointwise convergence,
\[
S_{\mu_n}(\ell)\longrightarrow e^{-\lambda}\sum_{u=0}^{\ell}(p\lambda)^u=s(\ell).
\]
Also, for all sufficiently large \(n\), by Lemma~\ref{lem:Lp-under-U},
\[
U\{L_p=\ell\}
=
\frac{1}{\ell!\,p^\ell}-\frac{1}{(\ell+1)!\,p^{\ell+1}},
\]
so
\[
b_{n,\ell}\longrightarrow
\left(\frac{1}{\ell!\,p^\ell}-\frac{1}{(\ell+1)!\,p^{\ell+1}}\right)\Phi\!\bigl(s(\ell)\bigr).
\]

By the domination hypothesis,
\[
0\le S_{\mu_n}(\ell)\le \sum_{u=0}^{\ell}B^u\le (\ell+1)B^\ell\le (2B)^\ell
\]
for all sufficiently large \(n\). Therefore, by \eqref{eq:Phi-growth} and \(U\{L_p=\ell\}\le 1/(\ell!\,p^\ell)\),
\[
|b_{n,\ell}|
\le
\frac{C_\Phi\bigl(1+(2B)^{d\ell}\bigr)}{\ell!\,p^\ell}
\]
for all sufficiently large \(n\). The right-hand side is summable in \(\ell\), so dominated convergence yields the result.
\end{proof}

\begin{remark}\label{rem:Phi-limit-rv}
Equivalently, the limiting right-hand side in Theorem~\ref{thm:Phi-top-m-asymp} is \(\E[\Phi(Z_{c,p})]\), where
\[
Z_{c,p}:=s(J_p),
\qquad
\Prob(J_p=\ell)=a_\ell:=
\frac{1}{\ell!\,p^\ell}
-
\frac{1}{(\ell+1)!\,p^{\ell+1}}
\qquad(\ell\ge0).
\]
Indeed, if \(t_\ell:=(\ell!\,p^\ell)^{-1}\), then
\[
a_\ell=t_\ell-t_{\ell+1}\ge0,
\qquad
\sum_{\ell=0}^{\infty}a_\ell=t_0-\lim_{\ell\to\infty}t_{\ell+1}=1.
\]
Thus, for every \(\Phi\) satisfying the hypotheses of Theorem~\ref{thm:Phi-top-m-asymp},
\[
\E[\Phi(Z_{c,p})]
=
\sum_{\ell=0}^{\infty}a_\ell\,\Phi(s(\ell)).
\]
\end{remark}

\begin{corollary}\label{cor:fp-series}
For every \(p\ge1\) and \(c\in\R\),
\[
f_p(c)=
\sum_{\ell=0}^{\infty}
\left(\frac{1}{\ell!\,p^\ell}-\frac{1}{(\ell+1)!\,p^{\ell+1}}\right)
\bigl(s(\ell)-1\bigr)_+.
\]
\end{corollary}

\begin{proof}
Fix \(m\ge1\), independent of \(n\), let \(k=k_n(c)\), and set \(\mu_n(n-u):=P_k^{(m)}(u)\). By Lemma~\ref{prop:Qm-mixture},
\[
M_n:=(Q^{(m)})^{*k}=\widetilde Q_{\mu_n,p}.
\]
By Proposition~\ref{lem:Pk-m-Poisson}, \(\mu_n(n-u)\to e^{-\lambda}\lambda^u/u!\) for each fixed \(u\ge0\).
Also, by Lemma~\ref{lem:Linfty-domination},
\[
0\le P_k^{(m)}(u)\,u!\,p^u\le \bigl(p\,n e^{-mk/n}\bigr)^u.
\]
Since
\[
p\,n e^{-mk/n}\longrightarrow p e^{-c},
\]
there exists \(B\ge1\) such that
\[
0\le P_k^{(m)}(u)\,u!\,p^u\le B^u
\]
for all sufficiently large \(n\) and all \(u\ge0\). Hence Theorem~\ref{thm:Phi-top-m-asymp} applies with
\[
\Phi(t)=(t-1)_+.
\]
By Proposition~\ref{prop:TV-likelihood},
\[
\sum_{x\in G_{n,p}}U(x)\Bigl(\frac{M_n(x)}{U(x)}-1\Bigr)_+
=\|M_n-U\|_{\TV}.
\]
Now Theorem~\ref{thm:profile-top-m} yields the stated formula.
\end{proof}

\begin{corollary}[Poisson-regime \texorpdfstring{$L^q(U)$}{Lq(U)} profile]\label{cor:Lq-top-m-asymp}
Let \(1\le q<\infty\). Under the hypotheses of Theorem~\ref{thm:Phi-top-m-asymp},
\[
\left\|\frac{M_n(\cdot)}{U(\cdot)}-1\right\|_{L^q(U)}^q
\longrightarrow
H_{q,p}(c),
\]
where
\[
H_{q,p}(c):=
\sum_{\ell=0}^{\infty}
\left(\frac{1}{\ell!\,p^\ell}-\frac{1}{(\ell+1)!\,p^{\ell+1}}\right)|s(\ell)-1|^q.
\]
\end{corollary}

\begin{proof}
Apply Theorem~\ref{thm:Phi-top-m-asymp} with \(\Phi(t)=|t-1|^q\).
\end{proof}

\begin{theorem}[Closed form for the \texorpdfstring{$\chi^2$}{chi-squared} profile]\label{thm:L2-top-m-limit}
Under the hypotheses of Theorem~\ref{thm:Phi-top-m-asymp},
\[
\chi^2(M_n,U)\longrightarrow H_{2,p}(c)=:g_p(c),
\]
where
\[
g_p(c)=
\begin{cases}
\displaystyle \frac{r+1}{r-1}e^{-\lambda(2-r)}-\frac{2}{r-1}e^{-\lambda}-1,
& r\neq1,\\[12pt]
\displaystyle (2\lambda+1)e^{-\lambda}-1,
& r=1.
\end{cases}
\]
\end{theorem}

\begin{proof}
By Corollary~\ref{cor:Lq-top-m-asymp},
\[
\chi^2(M_n,U)\longrightarrow H_{2,p}(c).
\]
Set
\[
H:=H_{2,p}(c),
\qquad
a_\ell:=\frac{1}{\ell!\,p^\ell}-\frac{1}{(\ell+1)!\,p^{\ell+1}},
\qquad
s_\ell:=e^{-\lambda}\sum_{u=0}^{\ell}r^u.
\]
Then
\[
H=\sum_{\ell=0}^{\infty}a_\ell(s_\ell-1)^2.
\]

Let
\[
t_u:=\frac{1}{u!\,p^u}.
\]
Since \(a_\ell=t_\ell-t_{\ell+1}\),
\[
\sum_{\ell=0}^{\infty}a_\ell=1,
\qquad
\sum_{\ell=u}^{\infty}a_\ell=t_u.
\]
Also, by Tonelli,
\[
\sum_{\ell=0}^{\infty}a_\ell s_\ell
=
e^{-\lambda}\sum_{u=0}^{\infty}r^u\sum_{\ell=u}^{\infty}a_\ell
=
e^{-\lambda}\sum_{u=0}^{\infty}\frac{r^u}{u!\,p^u}
=
e^{-\lambda}\sum_{u=0}^{\infty}\frac{\lambda^u}{u!}
=1.
\]
Hence
\[
H+1=\sum_{\ell=0}^{\infty}a_\ell s_\ell^2.
\]
Expanding \(s_\ell^2\) and applying Tonelli again,
\[
H+1
=
e^{-2\lambda}\sum_{u,v=0}^{\infty}r^{u+v}
\sum_{\ell=\max(u,v)}^{\infty}a_\ell
=
e^{-2\lambda}\sum_{u,v=0}^{\infty}
\frac{r^{u+v}}{\max(u,v)!\,p^{\max(u,v)}}.
\]
Split into diagonal and off-diagonal terms:
\[
H+1
=
e^{-2\lambda}\sum_{u=0}^{\infty}\frac{r^{2u}}{u!\,p^u}
+
2e^{-2\lambda}\sum_{v=1}^{\infty}\frac{r^v}{v!\,p^v}\sum_{u=0}^{v-1}r^u.
\]

If \(r\neq1\), then
\[
\sum_{u=0}^{v-1}r^u=\frac{r^v-1}{r-1},
\qquad
\frac{r^v}{v!\,p^v}=\frac{\lambda^v}{v!},
\qquad
\frac{r^{2u}}{u!\,p^u}=\frac{(p\lambda^2)^u}{u!}.
\]
Therefore
\begin{align*}
H+1
&=
e^{-2\lambda}e^{p\lambda^2}
+\frac{2e^{-2\lambda}}{r-1}\sum_{v=1}^{\infty}\frac{\lambda^v(r^v-1)}{v!}\\
&=
e^{-\lambda(2-r)}
+\frac{2e^{-2\lambda}}{r-1}\bigl(e^{\lambda r}-e^\lambda\bigr)\\
&=
\frac{r+1}{r-1}e^{-\lambda(2-r)}-\frac{2}{r-1}e^{-\lambda}.
\end{align*}
Subtracting \(1\) gives the stated formula.

If \(r=1\), then
\[
\sum_{u=0}^{v-1}r^u=v,
\qquad
\sum_{u=0}^{\infty}\frac{r^{2u}}{u!\,p^u}
=\sum_{u=0}^{\infty}\frac{\lambda^u}{u!}
=e^\lambda.
\]
Hence
\begin{align*}
H+1
&=
e^{-2\lambda}e^\lambda
+2e^{-2\lambda}\sum_{v=1}^{\infty}\frac{\lambda^v}{v!}\,v\\
&=
e^{-\lambda}
+2\lambda e^{-2\lambda}\sum_{w=0}^{\infty}\frac{\lambda^w}{w!}\\
&=
(2\lambda+1)e^{-\lambda}.
\end{align*}
Subtract \(1\).
\end{proof}

\begin{remark}\label{rem:Lq-explicit}
The cases \(q=1\) and \(q=2\) are exceptional:
\[
H_{1,p}(c)=2f_p(c),
\qquad
H_{2,p}(c)=g_p(c).
\]
For \(q=2\), the square expands into a double sum, and the \(\ell\)-sum collapses through the identity
\[
\sum_{\ell\ge\max(u,v)}U\{L_p=\ell\}
=\frac{1}{\max(u,v)!\,p^{\max(u,v)}}.
\]
No comparably simple closed form arises from this method for general real \(1<q<\infty\). For \textit{even integers} \(q\ge4\), we can expand \((s(\ell)-1)^q\) termwise, but the resulting formulas are cumbersome.
\end{remark}

\begin{definition}\label{def:hp-profile}
Define
\[
h_p(c):=
\sum_{\ell=0}^{\infty}
\left(\frac{1}{\ell!\,p^\ell}-\frac{1}{(\ell+1)!\,p^{\ell+1}}\right)
s(\ell)\log s(\ell).
\]
\end{definition}

\begin{corollary}[Relative-entropy profile]\label{cor:KL-top-m-asymp}
Under the hypotheses of Theorem~\ref{thm:Phi-top-m-asymp},
\[
D(M_n\|U)\longrightarrow h_p(c).
\]
\end{corollary}

\begin{proof}
Apply Theorem~\ref{thm:Phi-top-m-asymp} with \(\Phi(t)=t\log t\), where \(\Phi(0)=0\). Since
\[
|t\log t|\le t^2+e^{-1}\le 1+t^2
\qquad(t\ge0),
\]
the growth hypothesis holds with \(C_\Phi=1\) and \(d=2\).
\end{proof}

For \(p=1\) and \(m=1\), this recovers the cutoff-window entropy limit of \cite{Stark2002InformationLoss}.

We next compare relative entropy with total variation and \(\chi^2\).
\begin{proposition}\label{prop:KL-vs-TV-chisq}
For every probability measure $M$ on $G_{n,p}$,
\[
2\|M-U\|_{\TV}^2
\le
D(M\|U)
\le
\log\bigl(1+\chi^2(M,U)\bigr)
\le
\chi^2(M,U).
\]
\end{proposition}

\begin{proof}
By the Csisz\'ar--Kullback--Pinsker inequality, in the notation of
\cite[Eq.~(34), Thm.~31]{vanErvenHarremoes},
\[
\frac12\,V(M,U)^2\le D(M\|U),
\qquad
V(M,U)=2\|M-U\|_{\TV},
\]
hence
\[
2\|M-U\|_{\TV}^2\le D(M\|U).
\]
Also,
\[
D(M\|U)\le \log\bigl(1+\chi^2(M,U)\bigr)
\]
by \cite[Eq.~(16)]{NishiyamaSasonChi2}. Finally,
\[
\log\bigl(1+\chi^2(M,U)\bigr)\le \chi^2(M,U)
\]
because $\log(1+x)\le x$ for all $x\ge0$.
\end{proof}

\begin{corollary}\label{cor:KL-profile-comparison}
For every fixed \(p\ge1\) and \(c\in\R\),
\[
2f_p(c)^2
\le
h_p(c)
\le
\log\bigl(1+g_p(c)\bigr)
\le
g_p(c).
\]
\end{corollary}

\begin{proof}
Fix any \(m\ge1\), independent of \(n\), and set \(k=k_n(c)\). Proposition~\ref{lem:Pk-m-Poisson} and Lemma~\ref{lem:Linfty-domination} verify the hypotheses of Corollary~\ref{cor:KL-top-m-asymp} and Theorem~\ref{thm:L2-top-m-limit}. Apply Proposition~\ref{prop:KL-vs-TV-chisq} to \(M=(Q^{(m)})^{*k}\), then let \(n\to\infty\) and use Theorem~\ref{thm:profile-top-m}, Corollary~\ref{cor:KL-top-m-asymp}, and Theorem~\ref{thm:L2-top-m-limit}.
\end{proof}

\begin{proposition}[Monotonicity of separation, \texorpdfstring{$L^q(U)$}{Lq(U)}, and relative entropy]\label{prop:Lq-KL-monotone}
Let \(M\) be any probability measure on \(G_{n,p}\), and let \(P_M\) be the transition kernel of the right random walk with one-step law \(M\). Then, as functions of \(k\ge0\),
\[
\sep(M^{*k},U),\qquad
\left\|\frac{M^{*k}(\cdot)}{U(\cdot)}-1\right\|_{L^q(U)}
\quad(1\le q\le\infty),
\qquad
D(M^{*k}\|U)
\]
are nonincreasing.
\end{proposition}

\begin{proof}
The separation monotonicity is standard; see \cite[Exercise~6.4]{LPW}.

For \(1\le q<\infty\), apply the data-processing inequality for \(f\)-divergences to
\[
\mu=M^{*k},\qquad \nu=U,\qquad K=P_M,\qquad f(t)=|t-1|^q,
\]
where \(P_M(x,y)=M(x^{-1}y)\), so \(\mu K=M^{*(k+1)}\) and \(\nu K=U\).
Since \(UP_M=U\) and
\[
D_f(\mu\|U)
=
\sum_x U(x)\left|\frac{\mu(x)}{U(x)}-1\right|^q
=
\left\|\frac{\mu(\cdot)}{U(\cdot)}-1\right\|_{L^q(U)}^q,
\]
we get
\[
\left\|\frac{M^{*(k+1)}(\cdot)}{U(\cdot)}-1\right\|_{L^q(U)}^q
\le
\left\|\frac{M^{*k}(\cdot)}{U(\cdot)}-1\right\|_{L^q(U)}^q.
\]
The same argument with \(f(t)=t\log t\) gives
\[
D(M^{*(k+1)}\|U)\le D(M^{*k}\|U).
\]
See \cite[Theorem~14]{LieseVajda2006}.

For \(q=\infty\), define
\[
P_M^*(y,x):=\frac{U(x)P_M(x,y)}{U(y)}.
\]
Since \(U(x),U(y)>0\), \(P_M(x,y)\ge0\), and \(UP_M=U\),
\[
P_M^*(y,x)\ge0,
\qquad
\sum_x P_M^*(y,x)
=
\frac{1}{U(y)}\sum_x U(x)P_M(x,y)
=
\frac{(UP_M)(y)}{U(y)}
=
1,
\]
so \(P_M^*\) is a Markov kernel.

Also,
\[
M^{*(k+1)}=M^{*k}P_M,
\]
hence
\[
M^{*(k+1)}(y)=\sum_x M^{*k}(x)P_M(x,y).
\]
Dividing by \(U(y)\),
\[
\frac{M^{*(k+1)}(y)}{U(y)}
=
\sum_x \frac{U(x)P_M(x,y)}{U(y)}\,\frac{M^{*k}(x)}{U(x)}
=
\sum_x P_M^*(y,x)\frac{M^{*k}(x)}{U(x)}.
\]
Since \(\sum_x P_M^*(y,x)=1\),
\[
\frac{M^{*(k+1)}(y)}{U(y)}-1
=
\sum_x P_M^*(y,x)\left(\frac{M^{*k}(x)}{U(x)}-1\right).
\]
Therefore
\[
\left|\frac{M^{*(k+1)}(y)}{U(y)}-1\right|
\le
\sum_x P_M^*(y,x)\left|\frac{M^{*k}(x)}{U(x)}-1\right|
\le
\left\|\frac{M^{*k}(\cdot)}{U(\cdot)}-1\right\|_{L^\infty(U)},
\]
because \(P_M^*(y,\cdot)\) is a probability measure. Taking the maximum over \(y\) gives
\[
\left\|\frac{M^{*(k+1)}(\cdot)}{U(\cdot)}-1\right\|_{L^\infty(U)}
\le
\left\|\frac{M^{*k}(\cdot)}{U(\cdot)}-1\right\|_{L^\infty(U)}.\qedhere
\]
\end{proof}

\begin{corollary}[Monotonicity of the limit profiles]\label{cor:profile-monotone}
For every fixed \(p\ge1\), the functions
\[
c\longmapsto f_p(c),
\qquad
c\longmapsto s_p(c),
\qquad
c\longmapsto H_{q,p}(c)\quad(1\le q<\infty),
\qquad
c\longmapsto H_{\infty,p}(c),
\qquad
c\longmapsto h_p(c)
\]
are nonincreasing on \(\R\). In particular,
\[
c\longmapsto g_p(c)=H_{2,p}(c)
\]
is nonincreasing.
\end{corollary}

\begin{proof}
Fix \(m\ge1\), independent of \(n\). If \(c_1<c_2\), then
\[
k_n(c_1)\le k_n(c_2).
\]
Hence, by Proposition~\ref{prop:Lq-KL-monotone},
\[
\sep\bigl((Q^{(m)})^{*k_n(c)},U\bigr),\qquad
\left\|\frac{(Q^{(m)})^{*k_n(c)}(\cdot)}{U(\cdot)}-1\right\|_{L^q(U)},\qquad
D\bigl((Q^{(m)})^{*k_n(c)}\|U\bigr)
\]
are nonincreasing in \(c\). Since
\[
\bigl\|(Q^{(m)})^{*k_n(c)}-U\bigr\|_{\TV}
=
\frac12\left\|\frac{(Q^{(m)})^{*k_n(c)}(\cdot)}{U(\cdot)}-1\right\|_{L^1(U)},
\]
the total-variation term is also nonincreasing. For \(1\le q<\infty\), because \(x\mapsto x^q\) is increasing on \([0,\infty)\), the quantity
\[
\left\|\frac{(Q^{(m)})^{*k_n(c)}(\cdot)}{U(\cdot)}-1\right\|_{L^q(U)}^q
\]
is also nonincreasing in \(c\). The case \(q=\infty\) is already covered by the monotonicity of the \(L^\infty(U)\)-norm.

Proposition~\ref{lem:Pk-m-Poisson} and Lemma~\ref{lem:Linfty-domination} verify the hypotheses of Corollary~\ref{cor:Lq-top-m-asymp} and Corollary~\ref{cor:KL-top-m-asymp} for the top-\(m\) laws. Letting \(n\to\infty\) and applying Theorem~\ref{thm:profile-top-m}, Theorem~\ref{thm:sep-top-m}, Corollary~\ref{cor:Lq-top-m-asymp}, Corollary~\ref{cor:Linfty-top-m-limit}, and Corollary~\ref{cor:KL-top-m-asymp}, we obtain that
\[
f_p(c),\qquad
s_p(c),\qquad
H_{q,p}(c)\ (1\le q<\infty),\qquad
H_{\infty,p}(c),\qquad
h_p(c)
\]
are nonincreasing. Since \(g_p(c)=H_{2,p}(c)\), the same holds for \(g_p\).
\end{proof}

\section{Optimal strong stationary times for the reversed chain}\label{sec:sst}

This section uses the reversed chain to expose the frozen-tail structure behind separation. Conditional uniformity on the touched-label set gives optimal strong stationary times whose tails equal the separation distance.

We use the reversed-chain notation \(\check Q^{(m)},(Y_k),T_k\) fixed in the introduction. For \(S\subseteq[n]\), let \(F_S\) be the set of colored arrangements whose front block consists exactly of the touched labels \(S\), while the untouched labels \(S^c\) remain frozen in increasing order with color \(0\). That is,
\[
F_S:=\Bigl\{(s,\sigma)\in G_{n,p}:
\{\sigma(i):1\le i\le |S|\}=S,\ 
\sigma(i)<\sigma(j)\ \text{for all } |S|<i<j\le n,\ 
s_i=0\ \text{for all } |S|<i\le n
\Bigr\}.
\]

As we will see, we use the reversed chain because its inverse move preserves a frozen-tail structure: after time \(k\), the touched labels form the front block, while the untouched labels remain together in increasing order with color \(0\). Thus, conditional on \(T_k=S\), the chain is uniform on the simple set \(F_S\). This gives natural strong stationary times. By the reversal invariance from Section~\ref{sec:prel}, the resulting separation formulas also apply to the original colored top-\(m\)-to-random chain.

\begin{lemma}\label{lem:sst-conditional-uniform}
For every \(k\ge0\) and every \(S\subseteq[n]\) with \(\Prob(T_k=S)>0\),
\[
\Law(Y_k\mid T_k=S)=\Unif(F_S).
\]
\end{lemma}

\begin{proof}
Fix \(S\subseteq[n]\) with \(\Prob(T_k=S)>0\). On \(\{T_k=S\}\), no label in \(S^c\) has ever been selected, so the labels in \(S^c\) remain in increasing order and with color \(0\). If \(i\in S\), then at its last selection \(i\) is moved in front of every label in \(S^c\), and later moves preserve this relative order. Hence \(Y_k\in F_S\) on \(\{T_k=S\}\).

It remains to show that all states in \(F_S\) have the same conditional probability. Fix \(x,x'\in F_S\). Choose \(\rho\in\mathfrak S_S\) and \(a\in C_p^S\) sending the ordering and coloring of the front \(S\)-block of \(x\) to that of \(x'\). Apply \((\rho,a)\) to a history by replacing each selected label \(i\in S\) by \(\rho(i)\), and by replacing the corresponding assigned color \(c\) by \(c+a_i\). Since by definition the selected ordered lists and color assignments are uniform, this is a probability-preserving bijection from histories with
\[
T_k=S,\qquad Y_k=x
\]
to histories with
\[
T_k=S,\qquad Y_k=x'.
\]
Thus
\[
\Prob(Y_k=x\mid T_k=S)=\Prob(Y_k=x'\mid T_k=S)
\qquad(x,x'\in F_S).
\]
Since \(Y_k\in F_S\) on \(\{T_k=S\}\), the conditional law is uniform on \(F_S\).
\end{proof}

\begin{theorem}[Optimal strong stationary times for the reversed chain]\label{thm:sst-top-m}
For the reversed chain, if \(p\ge2\), then \(\tau_{\mathrm{cov}}:=\min\{k:T_k=[n]\}\) is an optimal strong stationary time and, for every \(k\ge0\),
\[
\Prob(\tau_{\mathrm{cov}}>k)
=
1-P_k^{(m)}(0)
=
\sep\bigl((Q^{(m)})^{*k},U\bigr).
\]
If \(p=1\), then \(\tau_{n-1}:=\min\{k:|T_k|\ge n-1\}\) is an optimal strong stationary time and, for every \(k\ge0\),
\[
\Prob(\tau_{n-1}>k)
=
1-P_k^{(m)}(0)-P_k^{(m)}(1)
=
\sep\bigl((Q^{(m)})^{*k},U\bigr).
\]
\end{theorem}

\begin{proof}
In each case, the displayed time is a stopping time, since the events \(\{\tau_{\mathrm{cov}}\le k\}\) and \(\{\tau_{n-1}\le k\}\) are determined by \(T_k\), hence by the first \(k\) moves. 

Write
\[
E_k:=n-|T_k|.
\]
Then \(E_k\) has law \(P_k^{(m)}\), since \(P_k^{(m)}(u)=\Prob(n-|T_k|=u)\).

Assume \(p\ge2\). 
Since \(\tau_{\mathrm{cov}}\le k\iff T_k=[n]\), Lemma~\ref{lem:sst-conditional-uniform} gives, for every \(y\in G_{n,p}\),
\begin{equation}\label{eq:sst-cov-uniform}
\Prob(Y_k=y,\tau_{\mathrm{cov}}\le k)
=
U(y)\Prob(\tau_{\mathrm{cov}}\le k),
\end{equation}
because \(F_{[n]}=G_{n,p}\); if \(\Prob(T_k=[n])=0\), both sides are \(0\).
Thus
\[
\Prob(\tau_{\mathrm{cov}}>k)
=
\Prob(E_k>0)
=
1-P_k^{(m)}(0).
\]
By the reversal identity from Section~\ref{sec:prel} and Theorem~\ref{thm:sep-top-m},
\[
\sep(\Law(Y_k),U)
=
\sep\bigl((Q^{(m)})^{*k},U\bigr)
=
1-P_k^{(m)}(0).
\]

Assume \(p=1\). If \(n=1\), then \(G_{1,1}\) is a singleton, and the claim is immediate. Assume \(n\ge2\). Since \(\tau_{n-1}\le k\iff |T_k|\ge n-1\), we consider the cases \(|T_k|=n\) and \(|T_k|=n-1\).

For every \(\pi\in S_n\), Lemma~\ref{lem:sst-conditional-uniform} gives
\begin{equation}\label{eq:sst-p1-full}
\Prob(Y_k=\pi,\ |T_k|=n)
=
U(\pi)\Prob(|T_k|=n),
\end{equation}
because \(F_{[n]}=S_n\); if \(\Prob(T_k=[n])=0\), both sides are \(0\).
By relabeling invariance,
\[
\Prob(T_k=[n]\setminus\{j\})
=
\frac1n\Prob(|T_k|=n-1).
\]
Hence, again by Lemma~\ref{lem:sst-conditional-uniform}, for every \(j\in[n]\),
\[
\Prob(Y_k=\pi,\ T_k=[n]\setminus\{j\})
=
\frac{\mathbf 1_{\{\pi(n)=j\}}}{(n-1)!}
\Prob(T_k=[n]\setminus\{j\}),
\]
with the identity trivial when \(\Prob(T_k=[n]\setminus\{j\})=0\). Summing over \(j\),
\begin{equation}\label{eq:sst-p1-one-missing}
\begin{aligned}
\Prob(Y_k=\pi,\ |T_k|=n-1)
&=
\frac{1}{(n-1)!}\Prob(T_k=[n]\setminus\{\pi(n)\})\\
&=
\frac1{n!}\Prob(|T_k|=n-1)
=
U(\pi)\Prob(|T_k|=n-1).
\end{aligned}
\end{equation}
Since
\[
\{\tau_{n-1}\le k\}
=
\{|T_k|=n\}\,\sqcup\,\{|T_k|=n-1\},
\]
adding \eqref{eq:sst-p1-full} and \eqref{eq:sst-p1-one-missing} gives
\begin{equation}\label{eq:sst-p1-uniform}
\Prob(Y_k=\pi,\tau_{n-1}\le k)
=
U(\pi)\Prob(\tau_{n-1}\le k).
\end{equation}
Also,
\[
\Prob(\tau_{n-1}>k)
=
\Prob(E_k\ge2)
=
1-P_k^{(m)}(0)-P_k^{(m)}(1).
\]
By the reversal identity from Section~\ref{sec:prel} and Theorem~\ref{thm:sep-top-m},
\[
\sep(\Law(Y_k),U)
=
\sep\bigl((Q^{(m)})^{*k},U\bigr)
=
1-P_k^{(m)}(0)-P_k^{(m)}(1).
\]

In either case, write \(\tau\) for the displayed stopping time. In both cases, \(\{\tau>k\}\subseteq\{E_k>0\}\). For each fixed label \(i\) and each \(k\ge1\),
\[
\Prob(i\notin T_k)=\left(1-\frac{m}{n}\right)^k.
\]
Thus, by the union bound,
\[
\Prob(\tau>k)
\le
\Prob(E_k>0)
\le
n\left(1-\frac{m}{n}\right)^k
\longrightarrow 0.
\]
Hence \(\tau<\infty\) almost surely.

By \eqref{eq:sst-cov-uniform} for \(p\ge2\) and \eqref{eq:sst-p1-uniform} for \(p=1\), the corresponding \(\tau\) satisfies
\[
\Prob(Y_k=y,\tau\le k)=U(y)\Prob(\tau\le k)
\qquad(k\ge0,\ y\in G_{n,p}).
\]
We now verify the strong-stationary-time condition. Let
\[
B_j:=\{\tau\le j\},
\qquad
B_{-1}:=\varnothing.
\]
Thus, for \(j\ge0\),
\[
\Prob(Y_j=y,B_j)=U(y)\Prob(B_j).
\]
Let \(\check P^{(m)}\) be the transition kernel of \(Y\). For \(j\ge1\),
\[
\begin{aligned}
\Prob(Y_j=y,B_{j-1})
&=
\sum_z \Prob(Y_j=y,\ Y_{j-1}=z,\ B_{j-1})  \\
&=
\sum_z \Prob(Y_{j-1}=z,\ B_{j-1})\check P^{(m)}(z,y)  \\
&=
\Prob(B_{j-1})\sum_z U(z)\check P^{(m)}(z,y)  \\
&=
U(y)\Prob(B_{j-1}),
\end{aligned}
\]
where the last equality uses stationarity of \(U\) for the reversed chain, and for \(j=0\) this is trivial because \(B_{-1}=\varnothing\).
Therefore
\[
\begin{aligned}
\Prob(Y_\tau=y,\tau=j)
&=
\Prob(Y_j=y,\tau=j)  \\
&=
\Prob(Y_j=y,\ B_j\setminus B_{j-1})  \\
&=
\Prob(Y_j=y,B_j)-\Prob(Y_j=y,B_{j-1})  \\
&=
U(y)\Prob(B_j)-U(y)\Prob(B_{j-1})  \\
&=
U(y)\Prob(\tau=j).
\end{aligned}
\]
Thus \(Y_\tau\sim U\) and \(Y_\tau\) is independent of \(\tau\). Hence \(\tau\) is a strong stationary time.

Finally, by \cite[Lemma~6.12]{LPW},
\[
\sep(\Law(Y_k),U)\le \Prob(\tau>k).
\]
Since equality holds for every \(k\), \(\tau\) is optimal.
\end{proof}

\begin{remark}
This is analogous to the classical top-to-random example of \cite[Example~6.15]{LPW}, which corresponds to \(p=1\) and \(m=1\).
\end{remark}

\begin{corollary}[Limit laws for the optimal stopping times]\label{cor:sst-limit-laws}
Fix \(m\ge1\), independent of \(n\). If \(p\ge2\), then
\[
\frac{m}{n}\tau_{\mathrm{cov}}-\log n
\Rightarrow \Lambda,
\]
where \(\Lambda\) has the standard Gumbel distribution
\[
\Prob(\Lambda\le c)=e^{-e^{-c}}.
\]
If \(p=1\), then
\[
\frac{m}{n}\tau_{n-1}-\log n
\Rightarrow \Lambda_1,
\]
where
\[
\Prob(\Lambda_1\le c)=e^{-e^{-c}}(1+e^{-c}).
\]
\end{corollary}

\begin{proof}
Let \(k=k_n(c)\) and \(\lambda=e^{-c}\). By Proposition~\ref{lem:Pk-m-Poisson},
\[
P_k^{(m)}(u)\to e^{-\lambda}\frac{\lambda^u}{u!}.
\]
Theorem~\ref{thm:sst-top-m} gives
\[
\Prob(\tau_{\mathrm{cov}}>k)=1-P_k^{(m)}(0)
\to 1-e^{-\lambda}
\qquad(p\ge2),
\]
and
\[
\Prob(\tau_{n-1}>k)=1-P_k^{(m)}(0)-P_k^{(m)}(1)
\to 1-e^{-\lambda}(1+\lambda)
\qquad(p=1).
\]
Substitute \(\lambda=e^{-c}\).

For either \(\tau=\tau_{\mathrm{cov}}\) or \(\tau=\tau_{n-1}\), since \(\tau\) is integer-valued and
\[
k_n(c)=\left\lfloor \frac{n}{m}(\log n+c)\right\rfloor,
\]
we have
\[
\left\{\frac{m}{n}\tau-\log n\le c\right\}
=
\{\tau\le k_n(c)\}.
\]
Thus the limiting tail probabilities above give the stated distribution functions.
\end{proof}

We finish by explaining how the theorems stated in the introduction follow from the results above.

\paragraph{\textbf{Proof of Theorems~\ref{thm:intro-LR}--\ref{thm:intro-topm}.}}
Theorem~\ref{thm:intro-LR} follows from Remark~\ref{rem:LR-general-mu}, Theorem~\ref{thm:Phi-general-mu}, Corollary~\ref{cor:sep-general-mu-explicit}, and Theorem~\ref{thm:Linfty-general-mu}.

Theorem~\ref{thm:intro-engine} follows from Theorem~\ref{prop:DFP3.1-colored}, Theorem~\ref{thm:sep-poisson-general}, Corollary~\ref{cor:Lq-top-m-asymp}, Theorem~\ref{thm:L2-top-m-limit}, Corollary~\ref{cor:KL-top-m-asymp}, and Theorem~\ref{thm:Linfty-poisson-general}.

Theorem~\ref{thm:intro-topm} follows from Lemma~\ref{prop:Qm-mixture}, Proposition~\ref{lem:Pk-m-Poisson}, Lemma~\ref{lem:Linfty-domination}, Theorem~\ref{thm:intro-engine}, Remark~\ref{rem:fp-endpoints}, Corollary~\ref{cor:profile-monotone}, and Theorem~\ref{thm:sst-top-m}.

\end{document}